\documentclass[12pt,a4paper,intlimits,sumlimits]{amsart}

\usepackage{amsmath, amsthm, mathtools}
\usepackage{amsfonts}
\usepackage[alphabetic]{amsrefs}
\usepackage{graphicx}
\usepackage{framed}
\usepackage{amssymb}
\usepackage{esvect}
\usepackage{xcolor}
\usepackage{eucal}
\usepackage{mathrsfs}
\usepackage{hyperref}
\usepackage[english]{babel}
\usepackage{mathrsfs}
\usepackage{esint}

\def \N {\mathbb{N}}
\def \R {\mathbb{R}}

\def \Ns {\mathscr{N}_{s}}
\def \Q {\mathcal Q}

\theoremstyle{definition}
\newtheorem{definition}{Definition}[section]
\newtheorem{example}[definition]{Example}
\newtheorem{remark}[definition]{Remark}

\theoremstyle{plain}
\newtheorem{theorem}[definition]{Theorem}
\newtheorem{proposition}[definition]{Proposition}
\newtheorem{lemma}[definition]{Lemma}
\newtheorem{corollary}[definition]{Corollary}

\numberwithin{equation}{section}

\usepackage{geometry}
\renewcommand{\epsilon}{\varepsilon}

\renewcommand{\leq}{\leqslant}
\renewcommand{\le}{\leqslant}
\renewcommand{\geq}{\geqslant}
\renewcommand{\ge}{\geqslant}

\allowdisplaybreaks

\title[Superposition of operators with Neumann conditions]{Some nonlinear problems\\
for the superposition of fractional operators\\ with Neumann boundary conditions}

\author[S. Dipierro, E. Proietti Lippi, C. Sportelli and E. Valdinoci]{Serena Dipierro, Edoardo Proietti Lippi, Caterina Sportelli and Enrico Valdinoci}

 \address{Department of Mathematics and Statistics
 \newline\indent University of Western Australia \newline\indent
 35 Stirling Highway, WA 6009 Crawley, Australia.\newline
 \newline\indent
\tt serena.dipierro@uwa.edu.au \newline\indent
\tt edoardo.proiettilippi@uwa.edu.au \newline\indent
\tt caterina.sportelli@uwa.edu.au \newline\indent
\tt enrico.valdinoci@uwa.edu.au}

\begin{document}

\maketitle

\begin{abstract}
We discuss the existence theory of a nonlinear problem of nonlocal type subject to Neumann boundary conditions. Differently from the existing literature, the elliptic operator under consideration is obtained as a superposition of operators of mixed order.

The setting that we introduce is very general and comprises, for instance, the sum of two fractional Laplacians, or of a fractional Laplacian and a Laplacian, as particular cases
(the situation in which there are infinitely many operators, and even a continuous distribution of operators, can be considered as well).

New bits of functional analysis are introduced to deal with this problem. An eigenvalue analysis divides the existence theory into two streams, one related to a Mountain Pass method, the other to a Linking technique.
\end{abstract}

\tableofcontents

\section{Introduction}
In this article, we consider a linear operator of mixed order obtained by the possible superposition of fractional Laplacians and the classical Laplacian. This operator is quite general and comprises the sum of two operators
of different order as a particular case. It also includes continuous superpositions of operators, as well as the case in which continuous and discrete superpositions are both present.

This operator is analyzed in the framework of a nonlinear equation in which the nonlinear source term is ``subcritical'' (of course, since the operator does not have a fixed order, this notion of ``criticality'' is somewhat delicate and needs to be specified with care).

The existence of solutions will be related to an eigenvalue problem and an appropriate analysis of the coefficient in front of the linear term of the source. Specifically, below a suitable eigenvalue the existence will be established through a Mountain Pass argument and above such a threshold by a Linking method.

Given the novelty of the topic under consideration, some new functional settings have to be introduced and the classical energy estimates need to be adapted to this new framework.
\medskip

The mathematical details of the problem under consideration go as follows.
Let $\mu$ be a nonnegative, finite (Borel) measure over $(0, 1)$.
Let $\alpha \ge 0$. The operator that we address here deals with the possible coexistence of a Laplacian and a superposition of fractional Laplacians, namely
\begin{equation}\label{ELLE}
L_{\alpha, \mu} (u):= -\alpha\Delta u +\int_{(0, 1)} (-\Delta)^s u\, d\mu(s).
\end{equation}
In a typical fashion, the notation $(-\Delta)^s$ stands for the fractional Laplacian, defined, for all~$s\in (0, 1)$, as 
\begin{equation}\label{deflaplacianofrazionario}
(- \Delta)^s\, u(x) = c_{N,s}\int_{\R^N} \frac{2u(x) - u(x+y)-u(x-y)}{|y|^{N+2s}}\, dy.
\end{equation}
The positive normalizing constant~$c_{N,s}$ is chosen in such a way to provide
consistent limits as~$s\nearrow1$ and as~$s\searrow0$, namely
\[
\lim_{s\nearrow1}(-\Delta)^su=(-\Delta)^1u=-\Delta u
\qquad{\mbox{and}}\qquad
\lim_{s\searrow0}(-\Delta)^s u=(-\Delta)^0u=u.
\]

Throughout the paper we denote by $\Omega\subset\R^N$ a bounded, Lipschitz domain.
The definition of ``nonlocal normal derivative'' associated with the
fractional Laplacian that we use in this paper has been introduced in~\cite{MR3651008} and reads as
$$ \Ns u(x) :=c_{N,s}\int_{\Omega}\frac{u(x)-u(y)}{|x-y|^{N+2s}}\,dy
\qquad \mbox{ for all }x\in \R^N \setminus \overline{\Omega}.$$
With this notation, we recall the definition of the Neumann
boundary conditions associated with the operator~$L_{\alpha, \mu}$ in~\eqref{ELLE} (these conditions were first introduced  in~\cite{MR4651677}
for the superposition of two operators, and then generalized to the setting
in~\eqref{ELLE} in~\cite{TUTTI}):

\begin{definition}\label{NDEFN}
We say that $u$ satisfies the $(\alpha,\mu)$-\textit{Neumann conditions} if
\begin{equation}\label{alfamucondition}
\begin{cases}
\displaystyle\int_{(0,1)} \Ns u(x)\,d\mu(s)=g\,\, \mbox{ for all }x\in \R^N \setminus \overline{\Omega},    & \mbox{ if }\alpha=0,
\\ \\
\partial_\nu u(x)=h \,\, \mbox{ for all } x\in \partial\Omega,
&  \mbox{ if } \mu\equiv 0,
\\ \\
\partial_\nu u(x)=h \,\, \mbox{ for all } x\in \partial\Omega
 \mbox{ and } \\ 
 \displaystyle\int_{(0,1)} \Ns u(x)\,d\mu(s)=g\,\, \mbox{ for all }x\in \R^N \setminus \overline{\Omega},  
 &\mbox{ if } \alpha\neq 0 \mbox{ and } \mu\not \equiv 0.
\end{cases}
\end{equation}

Moreover, when $g\equiv 0$ in $\R^N\setminus \overline{\Omega}$ 
and~$h\equiv 0$ on~$\partial \Omega$, we say that~\eqref{alfamucondition} are \textit{homogeneous $(\alpha,\mu)$-Neumann conditions}.
\end{definition}

It is worth noting that continuous superpositions of operators of different fractional orders have been recently considered in~\cite{DPSV, DPSV-P} in the case of a signed, finite Borel measure over $[0, 1]$. 
However, the setting considered in~\cite{DPSV, DPSV-P} turns out to be different from the one presented here, since it was tailored for critical problems subject to Dirichlet boundary conditions.

When $\alpha=0$, we need to identify a suitable value in the fractional range $(0, 1)$ which will play the role of a critical exponent (see the forthcoming Proposition~\ref{embeddings}).
In this case, since $\mu$ is nontrivial, there exists $\widetilde{s}_\sharp\in (0, 1)$ such that
\[
\mu([\widetilde{s}_\sharp, 1))>0.
\]
Hence, we define
\begin{equation}\label{ssharp}
s_\sharp:=\begin{cases}
\widetilde{s}_\sharp &\mbox{ if } \alpha =0,\\
1 &\mbox{ if } \alpha\neq 0.
\end{cases}
\end{equation}
Here below and in the rest of the paper we will use the notation
\[
2^*_{s_\sharp} = \frac{2N}{N-2s_\sharp}.
\]
The exponent~$2^*_{s_\sharp}$ plays the role of a ``critical exponent'' for a suitable nonlinear nonlocal problem which we now present in detail.\medskip

Let $\widetilde\lambda\in\R$. Here we address the study of the problem
\begin{equation}\label{problemasottocritico.pre}
\begin{cases}
L_{\alpha,\mu}(u)=\widetilde \lambda u +f(x,u) \qquad \mbox{in }\Omega
\\
\mbox{with homogeneous } (\alpha,\mu)\mbox{-Neumann conditions.}
\end{cases}
\end{equation}
In this setting, it is convenient to consider an equivalent formulation. Namely, setting
\begin{equation}\label{lambdadefn}
\lambda:=\widetilde\lambda+1,
\end{equation}
we consider the problem
\begin{equation}\label{problemasottocritico}
\begin{cases}
L_{\alpha,\mu}(u)+u=\lambda u +f(x,u)  \qquad \mbox{in }\Omega
\\
\mbox{with homogeneous } (\alpha,\mu)\mbox{-Neumann conditions.}
\end{cases}
\end{equation}
By~\cite[Theorem~1.5]{TUTTI} there exists a sequence of eigenvalues
\begin{equation}\label{sequenceLalfamu}
0=\widetilde\lambda_1 < \widetilde\lambda_2\le\widetilde\lambda_3\dots
\end{equation}
associated with the operator $L_{\alpha,\mu}$ under $(\alpha,\mu)$-Neumann conditions. Then, we will use the notation 
\begin{equation}\label{eigennotation}
\lambda_k:=\widetilde \lambda_k+1
\end{equation}
to identify the sequence of eigevanlues of the operator $L_{\alpha,\mu}+Id$ under $(\alpha,\mu)$-Neumann conditions. In particular, by~\eqref{sequenceLalfamu} it follows that
\begin{equation}\label{primouguale1}
\lambda_1 = 1.
\end{equation}
\medskip

The main goal of this article is to construct solutions of~\eqref{problemasottocritico.pre},
or equivalently of~\eqref{problemasottocritico}, for suitable parameter ranges of~$\widetilde\lambda$, or equivalently of~$\lambda$,
and appropriate
assumptions on the nonlinearity.

More explicitly,
throughout the paper we assume that the nonlinear term~$f:\Omega\times \R \to \R$ in~\eqref{problemasottocritico} is a Carath\'eodory function which satisfies a suitable subcritical growth, namely
\begin{equation}\label{AR1}
\begin{split}
&\mbox{there exist some constants $a_1,a_2>0$ and $p\in (2, 2^*_{s_\sharp})$ such that} \\
&|f(x,t)|\leq a_1+a_2|t|^{p-1} \quad\mbox{for any } t\in \R \mbox{ and for a.e. } x\in \Omega,\\
\end{split}
\end{equation}
and
\begin{equation}\label{AR2}
\lim_{t\to 0}\frac{f(x,t)}{t}=0  \quad \mbox{ uniformly for a.e. } x\in \Omega.
\end{equation}
Moreover, setting
\begin{equation}\label{definizioneF}
F(x,t):=\int_0^t f(x,\tau)\,d\tau,
\end{equation}
we assume that the Ambrosetti--Rabinowitz condition holds true, i.e.
\begin{equation}\label{AR3}
\begin{split}
&\mbox{there exist $\vartheta>2$ and $r\geq 0$ such that}\\
&0<\vartheta F(x,t)\leq f(x,t)t  \quad\mbox{for any } t \mbox{ such that } |t|> r \mbox{ and for a.e. } x\in \Omega.
\end{split}
\end{equation}
In addition, we suppose that 
\begin{equation}\label{AR4}
\begin{split}
&\mbox{there exist some constants~$\widetilde{\vartheta}>2$, $a_3>0$ and a function~$a_4\in L^1(\Omega)$ such that}\\
&F(x,t)\geq a_3|t|^{\widetilde{\vartheta}}-a_4(x).
\end{split}
\end{equation}
Furthermore, being $r$ as in~\eqref{AR3}, if~$r>0$, we ask that
\begin{equation}\label{AR5}
F(x,t)\geq 0  \quad\mbox{for any } t\in \R \mbox{ and for a.e. } x\in \Omega.
\end{equation}
We point out that condition~\eqref{AR4} was introduced in~\cite{addendum}
to complete the Ambrosetti-Rabinowitz condition in the presence
of Carath\'eodory functions.

Depending on the location of the parameter~$\lambda$ introduced in~\eqref{lambdadefn}, to prove the existence of a nontrivial solution for the problem~\eqref{problemasottocritico}, we will use either the Mountain Pass Theorem or the Linking Theorem.

To be more precise, on the one hand, when $\lambda\le\lambda_1=1$ (recall~\eqref{primouguale1}) we will show that problem~\eqref{problemasottocritico} has a Mountain Pass geometry.

On the other hand, we address the case $\lambda\ge\lambda_1=1$ as well
through a different method, and specifically
by using a Linking argument. 

We point out that the use of a Linking argument requires a suitable decomposition of the ambient space into the direct sum of closed subspaces. For this reason, when dealing with the case $\lambda\ge\lambda_1$, we restrict ourselves to a suitable subspace which guarantees such a feature. In particular, by using~\cite[Theorem~1.5]{TUTTI}, we are able to provide the existence of a complete orthogonal system of eigenfunctions for such a subspace and then the desired direct sum decomposition. We refer the reader to the definition in~\eqref{tildeacca} and to Section~\ref{sec-link} for a detailed discussion.
Moreover, it is worth noting that the choice of the subspace we have opted for is not the only option. We discuss another possible choice and the associated difficulties in Section~\ref{sec-app}.
\medskip

In our existence results, we will deal with weak solutions of~\eqref{problemasottocritico}, see Definitions~\ref{weakdefn} and~\ref{wdlinking} for the precise setting.

\begin{theorem}\label{MPT}
Let $f:\Omega\times \R \to \R$ satisfy~\eqref{AR1}, \eqref{AR2},  \eqref{AR3} and~\eqref{AR4}.

Then, for any $\lambda<\lambda_1$ there exists a nontrivial weak solution of problem~\eqref{problemasottocritico},
according to Definition~\ref{weakdefn}.
\end{theorem}

\begin{theorem}\label{LINK}
Let $f:\Omega\times \R \to \R$ satisfy~\eqref{AR1}, \eqref{AR2},  \eqref{AR3}, \eqref{AR4} and~\eqref{AR5}. 

Then, for any $\lambda\ge\lambda_1$ there exists a nontrivial weak
solution of problem~\eqref{problemasottocritico} according to
Definition~\ref{wdlinking}.
\end{theorem}

Our main results can be compared with the ones stated in~\cite[Theorems~1 and~2]{MR3002745}. Nonetheless, we emphasize that our setting presents different boundary conditions and is more general than the one introduced in~\cite{MR3002745}, not only for the wide generality of the operator~$L_{\alpha, \mu}$ under consideration, but also for the assumptions on the nonlinearity~$f(x, t)$ (e.g. one can compare~\eqref{AR3}, \eqref{AR4} and~\eqref{AR5} here with~\cite[formulas~(1.11), (1.14) and~(1.15)]{MR3002745}).

We emphasize that Theorems~\ref{MPT} and~\ref{LINK} turn out to be not only new in their broad generality, but they also allow us to address many specific cases which are new as well. 

In particular:
\begin{itemize}
\item Let $\beta>0$ and $\mu:=\beta\delta_s$,
being~$\delta_s$ the Dirac's delta at~$s\in(0,1)$. Hence, the operator~$L_{\alpha, \mu}$ introduced in~\eqref{ELLE} boils down to the mixed operator
\[
-\alpha\Delta + \beta(-\Delta)^s
\]
for some~$s\in(0,1)$,  which has been introduced, with Neumann boundary conditions, in~\cite{MR4438596, MR4651677}.

In this setting, to the best of our knowledge, our results happen to be new as well.  We treat this case in detail in Corollary~\ref{cor1}.

Moreover, if $\alpha=0$ and $\beta=1$, then problem 
\eqref{problemasottocritico} turns out to be a particular case
of the one addressed in~\cite[Section~5]{MPPL}.\medskip

\item The case of the superposition of a finite number of nonlocal operators with different orders, corresponding to the choice of the measure
\[
\mu:=\sum_{k=1}^n \delta_{s_k}
\]
and to an operator of the type
\[
-\alpha\Delta+\sum_{k=1}^n (-\Delta)^{s_k}
\]
for some $s_1,\dots, s_n\in(0,1)$, with~$n\ge 2$,
to our best knowledge is new as well in both the cases~$\alpha\neq 0$ and~$\alpha=0$.

The application of Theorems~\ref{MPT} and~\ref{LINK} for this choice of~$\mu$ will be discussed in Corollary~\ref{cor2}.\medskip

\item The case of a convergent series, namely
\[
\mu:= \sum_{k=0}^{+\infty} c_k \,\delta_{s_k}, \qquad{\mbox{where }}\, \sum_{k=0}^{+\infty} c_k\in(0,+\infty),
\]
with $c_k\ge0$ for all $k\in\N$, corresponding to the operator
\[
-\alpha\Delta + \sum_{k=0}^{+\infty} c_k (-\Delta)^{s_k},
\]
is also new in both the cases $\alpha\neq 0$ and $\alpha=0$.

We discuss this case in Corollary~\ref{cor3}.
\medskip

\item The continuous superposition of fractional operators of the form
\[
\int_0^1 \omega(s)(-\Delta)^s u \,ds,
\]
where $\omega$ is a measurable, nonnegative and non identically zero function, turns out to be new as well. 

We address this case in Corollary~\ref{cor4}.
\end{itemize}

\medskip

The rest of the paper is organized as follows. In Section~\ref{sec:func}
we introduce the functional analytical setting that we work in.
Sections~\ref{sec:mp} and~\ref{sec-link} are devoted
to the proofs of Theorems~\ref{MPT} and~\ref{LINK}, respectively.
Some explicit examples and applications depending on the  specific
choice of the measure~$\mu$ are provided in Section~\ref{sec:appex}.

The paper ends with Appendix~\ref {sec-app}, where
we motivate the specific choice of the functional space
used to employ the Linking Theorem.

\section{Functional setting}\label{sec:func}

In this section we introduce a functional analytical setting that is inspired by the one presented in~\cite[Section~2]{TUTTI} and allows us
to address the study of problem~\eqref{problemasottocritico} by using a variational argument.

Let $\Q:=\R^{2N}\setminus (\R^N\setminus \Omega)^2$. We denote the Gagliardo seminorm of $u$ as
\[
[u]_s:= \left({c_{N, s}}\iint_{\mathcal Q} \frac{|u(x)-u(y)|^2}{|x-y|^{N+2s}} \, dx\, dy\right)^{1/2}.
\]
Moreover, we define the functional space
\begin{equation}\label{HALPHAMU}
\mathcal H_{\alpha, \mu} (\Omega):= \begin{cases}
H^1(\Omega) &\hbox{ if } \mu \equiv 0,\\
\mathcal H_\mu (\Omega) &\hbox{ if } \alpha = 0,\\
H^1(\Omega)\cap \mathcal H_\mu (\Omega) &\hbox{ if }  \alpha\neq 0  \text{ and }
\mu \not\equiv 0,
\end{cases}
\end{equation}
endowed with the norm
\begin{equation}\label{normaalfamu}
\begin{split}
\|u\|_{\alpha, \mu} := \Bigg(\|u\|^2_{L^2(\Omega)} + \||h|^{1/2} u\|^2_{L^2(\partial\Omega)} &+ \| |g|^{1/2} u\|^2_{L^2(\R^N\setminus\Omega)} \\
&+ \alpha \|\nabla u\|^2_{L^2(\Omega)} +\frac12 \int_{(0, 1)} [u]_s^2 \, d\mu(s) \Bigg)^{1/2}.
\end{split}
\end{equation}
In~\eqref{HALPHAMU}, the space $\mathcal H_\mu (\Omega)$ is defined
as
\begin{equation}\label{Hmu}
\mathcal H_\mu (\Omega):= \left\lbrace u:\R^N\to\R \text{ measurable } : \|u\|_\mu <+\infty \right\rbrace,
\end{equation}
with the norm
\begin{equation}\label{normamu}
\|u\|_\mu := \left(\|u\|^2_{L^2(\Omega)} + \| |g|^{1/2} u\|^2_{L^2(\R^N\setminus\Omega)} +\frac12\int_{(0, 1)} [u]_s^2 \, d\mu(s) \right)^{1/2}.
\end{equation}

We can now recall the embedding results for the space $\mathcal H_{\alpha, \mu} (\Omega)$ stated in~\cite[Proposition~2.4 and Corollary~2.5]{TUTTI}:

\begin{proposition}\label{embeddings}
Let $s_\sharp$ be as in~\eqref{ssharp}. Then, there exists a positive constant~$c=c(N, \Omega, s_\sharp)$ such that, for any~$u\in\mathcal H_{\alpha,\mu}(\Omega)$, 
\[
\|u\|_{H^{s_\sharp}(\Omega)}\le c \|u\|_{\alpha, \mu},
\]
namely the space $\mathcal H_{\alpha,\mu}(\Omega)$ is continuosly embedded in $L^{2^*_{s_\sharp}}(\Omega)$. 

In addition, $\mathcal H_{\alpha,\mu}(\Omega)$ is compactly embedded in $L^p(\Omega)$ for any $p\in [1,  2^*_{s_\sharp})$.
\end{proposition}

We point out that, since we deal with homogeneous $(\alpha,\mu)$-Neumann boundary conditions in~\eqref{problemasottocritico},
we will assume from now on that~$h\equiv0$ and~$g\equiv0$.

To develop a Linking argument,  we introduce the subspace
\begin{equation}\label{tildeacca}
\widetilde{\mathcal H}_{\alpha, \mu}(\Omega):=\begin{cases}
H^1(\Omega) &\mbox{ if } \mu\equiv 0,\\
\widetilde{\mathcal H}_{\mu}(\Omega) &\mbox{ if } \alpha = 0,\\
 L^2(\R^N) \cap
\mathcal H_{\alpha, \mu}(\Omega)&\mbox{ if } \alpha\neq 0 \text{ and }\mu\not\equiv 0,
\end{cases}
\end{equation}
where
\begin{equation}\label{H0MU}
\widetilde{\mathcal H}_{\mu}(\Omega):=\left\{u\in\mathcal H_{\mu}(\Omega) : \int_{(0, 1)}\Ns u\, d\mu(s)=0
\mbox{ in }\R^N\setminus \overline{\Omega}\right\}.
\end{equation}
We endow the space $\widetilde{\mathcal H}_{\alpha, \mu}(\Omega)$ with the norm defined in~\eqref{normaalfamu}.
\begin{remark}
We observe that there exist functions
in~${\mathcal H}_{0, \mu}(\Omega)={\mathcal H}_{\mu}(\Omega)$ that do not belong to~$\widetilde{\mathcal H}_{0, \mu}(\Omega)
=\widetilde{\mathcal H}_{\mu}(\Omega)$.
Indeed, we consider $\Omega=(-1,1)$ and take~$u:\R\to \R$ such that
\begin{eqnarray*}
&&u(x)=x \quad \mbox{ if } x\in(-1,1)
\\
{\mbox{and }}&& 
\int_{(0, 1)}\Ns u(x)\, d\mu(s)=0\quad \mbox{ if } x\in\R\setminus[-1,1].
\end{eqnarray*}
In this way, we have that~$u\in \widetilde{\mathcal H}_{0, \mu}(\Omega)$ and, for all~$x\in\R\setminus[-1,1]$,
\[
u(x)=
\dfrac{\displaystyle \int_{(0,1)}\int_{-1}^1 \frac{y}{|x-y|^{1+2s}}\,dy\,d\mu(s)}{\displaystyle \int_{(0,1)}\int_{-1}^1 \frac{dy}{|x-y|^{1+2s}}\,d\mu(s)}.
\]
In particular, if $x<-1$, we have that, for any~$s\in(0,1)$,
\[
\int_{-1}^1 \frac{y}{|x-y|^{1+2s}}\,dy\leq 0 ,
\]
and thus $u(x)\leq 0$ if $x\leq 0$.

Now, if we consider the 
positive part of $u$, we have $u^+\in {\mathcal H}_{0, \mu}(\Omega)$ and $u^+(x)=0$ if $x\leq 0$. 
However, for any $x<-1$, we can compute
\begin{equation*}
\int_{(0, 1)}\Ns u^+(x)\, d\mu(s)=\int_{(0, 1)} \int_{-1}^1 \frac{u^+(x)-u^+(y)}{|x-y|^{1+2s}}\,dy\,d\mu(s)
=\int_{(0, 1)}\int_0^1 \frac{-u^+(y)}{|x-y|^{1+2s}}\,dy\,d\mu(s)
<0,
\end{equation*}
and therefore $u^+\not \in \widetilde{\mathcal H}_{0, \mu}(\Omega)$.
\end{remark}

We notice that, by the definition in~\eqref{tildeacca},
if~$\alpha\neq 0$ and~$\mu\not\equiv 0$ we have that the space~$\widetilde{\mathcal H}_{\alpha, \mu}(\Omega)$ is a closed subspace of~$L^2(\R^N)$. Moreover, if~$\mu\equiv 0$, we have that~$\widetilde{\mathcal H}_{\alpha, \mu}(\Omega)$ is a closed subspace of~$L^2(\Omega)$.
Since both~$L^2(\R^N)$ and~$L^2(\Omega)$ are separable, we infer that there exists a complete orthogonal system for~$\widetilde{\mathcal H}_{\alpha, \mu}(\Omega)$ in these cases.

Instead, when $\alpha=0$, we employ the result stated in~\cite[Theorem~1.5]{TUTTI}, according to which there exists a sequence of eigenfunctions~$e_k$ in~${\mathcal H}_{0, \mu}(\Omega)$ associated with the operator~$L_{0,\mu}$ which is a complete orthogonal system in~$L^2(\Omega)$. 

We also remark that, by construction, for any~$k\in\N\setminus\{0\}$
and any~$x\in\R^N\setminus \overline\Omega$,
\begin{equation}\label{dhuwo4ty4tgufkjfbewksbefi4yoi}
\int_{(0, 1)}\Ns e_k(x) d\mu(s) =0.
\end{equation}
This entails that the eigenfunctions~$e_k$ actually
belong to~$\widetilde{\mathcal H}_{0, \mu}(\Omega)$.

As a matter of fact, we have the following result:

\begin{proposition}\label{sistemacompleto}
The sequence $e_k$ provides a complete orthogonal system in $\widetilde{\mathcal H}_{0, \mu}(\Omega)$.
\end{proposition}

\begin{proof}
We first observe that, by construction,
for any~$k$, $j\in\N\setminus\{0\}$,
\[
\int_{(0,1)}\frac{c_{N, s}}{2}\iint_\Q\frac{(e_k(x)-e_k(y))(e_j(x)-e_j(y))}{|x-y|^{N+2s}}\,dx\,dy\,d\mu(s)=\lambda_k \int_\Omega e_k(x) e_j(x)\,dx.
\]
As a consequence, since~$e_k$ is an orthogonal system in~$L^2(\Omega)$, we have that it is orthogonal in~$\widetilde{\mathcal H}_{0, \mu}(\Omega)$ as well.

We now focus on proving the completeness of the system. To this aim,
let~$u\in \widetilde{\mathcal H}_{0, \mu}(\Omega)$. Since~$e_k$ is a complete system in~$L^2(\Omega)$, by considering the restriction of~$u$ over~$\Omega$, there exist coefficients~$a_k\in\R$ such that
\[
u=\sum_{k=1}^{+\infty}a_ke_k \quad \mbox{a.e. in } \Omega.
\]
Moreover, since $u\in \widetilde{\mathcal H}_{0, \mu}(\Omega)$, for any~$x\in \R^N\setminus \overline{\Omega}$ we have that
\[
u(x)=
\dfrac{\displaystyle \int_{(0,1)}\int_\Omega \frac{u(y)}{|x-y|^{N+2s}}\,dy\,d\mu(s)}{\displaystyle \int_{(0,1)}\int_\Omega \frac{dy}{|x-y|^{N+2s}}\,d\mu(s)}
=\dfrac{\displaystyle \int_{(0,1)}\int_\Omega \frac{\sum_{k=1}^{+\infty}a_ke_k(y)}{|x-y|^{N+2s}}\,dy\,d\mu(s)}{\displaystyle \int_{(0,1)}\int_\Omega \frac{dy}{|x-y|^{N+2s}}\,d\mu(s)}.
\]
Hence, by the Dominated Convergence Theorem, and
recalling~\eqref{dhuwo4ty4tgufkjfbewksbefi4yoi}, we get
\[
u(x)=\sum_{k=1}^{+\infty}a_k\;\dfrac{\displaystyle \int_{(0,1)}\int_\Omega \frac{e_k(y)}{|x-y|^{N+2s}}\,dy\,d\mu(s)}{\displaystyle \int_{(0,1)}\int_\Omega \frac{dy}{|x-y|^{N+2s}}\,d\mu(s)}
=\sum_{k=1}^{+\infty}a_k e_k(x),
\]
as desired.
\end{proof}

We now prove that the subspace $\widetilde{\mathcal H}_{0, \mu}(\Omega)$ is closed with 
respect to the weak convergence.

\begin{proposition}\label{chiusoconvergenzadebole}
The space $\widetilde{\mathcal H}_{0, \mu}(\Omega)$ is closed
with respect to the weak convergence in 
$\widetilde{\mathcal H}_{0, \mu}(\Omega)$.
\end{proposition}

\begin{proof}
Let $u_k$ be a sequence in $\widetilde{\mathcal H}_{0, \mu}(\Omega)$ which converges weakly to some $u$. We aim to prove that $u\in \widetilde{\mathcal H}_{0, \mu}(\Omega)$.

We observe that $u_k$ converges weakly in ${\mathcal H}_{0, \mu}(\Omega)$ and $u\in{\mathcal H}_{0, \mu}(\Omega)$.
Then, by the compact embedding stated in Proposition~\ref{embeddings},
$u_k$ converges strongly to $u$ in $L^p(\Omega)$ for any 
$p\in [1,2^*_{s_{\sharp}})$. Moreover, $u_k$ converges to $u$ 
a.e. in $\Omega$ and $|u_k|\leq h$ for some $h\in L^p(\Omega)$
(see~\cite[Theorem~IV.9]{brezis}).
Then, recalling that $u_k\in \widetilde{\mathcal H}_{0, \mu}(\Omega)$, if $x\in\R^N\setminus \overline{\Omega}$,
we can apply the Dominated Convergence Theorem to get
\begin{equation}\label{vdefn}
\lim_{k\to +\infty}u_k(x)=
\lim_{k\to +\infty}
\dfrac{\displaystyle \int_{(0,1)}\int_\Omega \frac{u_k(y)}{|x-y|^{N+2s}}\,dy\,d\mu(s)}{\displaystyle \int_{(0,1)}\int_\Omega \frac{dy}{|x-y|^{N+2s}}\,d\mu(s)}
=
\dfrac{\displaystyle \int_{(0,1)}\int_\Omega \frac{u(y)}{|x-y|^{N+2s}}\,dy\,d\mu(s)}{\displaystyle \int_{(0,1)}\int_\Omega \frac{dy}{|x-y|^{N+2s}}\,d\mu(s)}=:v(x).
\end{equation}

We claim that 
\begin{equation}\label{claimuv}
u=v \quad\mbox{ a.e. in } \R^N\setminus\overline\Omega.
\end{equation}
To this aim, let $\varphi\in C_c^\infty(\R^N\setminus \overline{\Omega}, [0,+\infty))$. Since~$u_k$ converges weakly to~$u$
in~${\mathcal H}_{0, \mu}(\Omega)$, we have that
\begin{equation}\label{limiteuv1}
\lim_{k\to +\infty}\int_{(0,1)}\frac{c_{N, s}}{2}\iint_\Q\frac{\big( (u_k-u)(x)- (u_k-u)(y)\big)(\varphi(x)-\varphi(y))}{|x-y|^{N+2s}}\,dx\,dy\,d\mu(s)=0.
\end{equation}
We write $\Q = \big(\Omega\times\Omega\big)\cup\big(\Omega\times(\R^N\setminus {\Omega})\big)\cup\big((\R^N\setminus  {\Omega})\times\Omega\big)$.  Since~$\varphi=0$ in~$\Omega$, 
\begin{equation}\label{limiteuv2}
\int_{(0,1)}\frac{c_{N, s}}{2}\iint_{\Omega\times\Omega}\frac{\big( (u_k-u)(x)- (u_k-u)(y)\big)(\varphi(x)-\varphi(y))}{|x-y|^{N+2s}}\,dx\,dy\,d\mu(s)=0.
\end{equation}
Moreover,  we have
\begin{equation}\label{limiteuv3}
\begin{split}
&\int_{(0,1)}\frac{c_{N, s}}{2}\iint_{\Omega\times(\R^N\setminus  {\Omega})}\frac{\big( (u_k-u)(x)- (u_k-u)(y)\big)(\varphi(x)-\varphi(y))}{|x-y|^{N+2s}}\,dx\,dy\,d\mu(s)\\
&= \int_{(0,1)}\frac{c_{N, s}}{2}\iint_{\Omega\times(\R^N\setminus  {\Omega})}\frac{(u_k-u)(y)\,\varphi(y)}{|x-y|^{N+2s}}\,dx\,dy\,d\mu(s)\\
&\qquad - \int_{(0,1)}\frac{c_{N, s}}{2}\iint_{\Omega\times(\R^N\setminus  {\Omega})}\frac{(u_k-u)(x)\,\varphi(y)}{|x-y|^{N+2s}}\,dx\,dy\,d\mu(s)=: I_1 - I_2.
\end{split}
\end{equation}
We start considering $I_2$.  We have
\[
\begin{split}
I_2&=\int_\Omega(u_k(x) - u(x)) \left(\;\int_{(0,1)}\frac{c_{N, s}}{2}\int_{\R^N\setminus \overline{\Omega}} \frac{\varphi(y)}{|x-y|^{N+2s}} \, dy\, d\mu(s) \right) dx.
\end{split}
\]
Since $\mbox{supp}(\varphi)\subset \R^N\setminus \overline{\Omega}$, we have
\[
\frac{c_{N, s}\varphi(y)}{|x-y|^{N+2s}} \ \mbox{ is bounded in } (0, 1)\times\mbox{supp}(\varphi),
\]
which entails that
\[
|I_2|\le C \int_\Omega |u_k(x) - u(x)|\, dx,
\]
for some $C=C(\varphi, \Omega, \mu)>0$. This gives that
\begin{equation}\label{limiteuv4}
\lim_{k\to +\infty}  \int_{(0,1)}\frac{c_{N, s}}{2}\iint_{\Omega\times(\R^N\setminus {\Omega})}\frac{(u_k-u)(x)\,\varphi(y)}{|x-y|^{N+2s}}\,dx\,dy\,d\mu(s) =\lim_{k\to+\infty} I_2 =0.
\end{equation}

Furthermore, 
\[
\begin{split}
I_1 &= \int_{\R^N\setminus \overline{\Omega}} (u_k-u)(y) \varphi(y) \left(  \int_{(0,1)}\frac{c_{N, s}}{2}\int_{\Omega} \frac{1}{|x-y|^{N+2s}} \, dx\, d\mu(s) \right) dy\\
&= \int_{\R^N\setminus \overline{\Omega}} C_\Omega (y) (u_k-u)(y) \varphi(y)\, dy,
\end{split}
\]
where $C_\Omega (y)>0$ is bounded on the support of $\varphi$ (in light of the fact that $\mbox{supp}(\varphi)\subset \R^N\setminus \overline{\Omega}$).
By \eqref{vdefn} and the Dominated Convergence Theorem, we get
\begin{equation}\label{limiteuv5}\begin{split}
&\lim_{k\to +\infty} \int_{(0,1)}\frac{c_{N, s}}{2}\iint_{\Omega\times(\R^N\setminus \overline{\Omega})}\frac{(u_k-u)(y)\,\varphi(y)}{|x-y|^{N+2s}}\,dx\,dy\,d\mu(s)=\lim_{k\to +\infty}I_1\\&\qquad\qquad= \int_{\R^N\setminus \overline{\Omega}} C_\Omega (y) (v-u)(y) \varphi(y)\, dy.
\end{split}\end{equation}

Moreover, exchanging the role of~$x$ and~$y$, we observe that
\begin{equation}\label{limiteuv6}
\begin{split}
&\int_{(0,1)}\frac{c_{N, s}}{2}\iint_{\Omega\times(\R^N\setminus \overline{\Omega})}\frac{\big( (u_k-u)(x)- (u_k-u)(y)\big)(\varphi(x)-\varphi(y))}{|x-y|^{N+2s}}\,dx\,dy\,d\mu(s)\\
&=\int_{(0,1)}\frac{c_{N, s}}{2}\iint_{(\R^N\setminus \overline{\Omega})\times\Omega}\frac{\big( (u_k-u)(x)- (u_k-u)(y)\big)(\varphi(x)-\varphi(y))}{|x-y|^{N+2s}}\,dx\,dy\,d\mu(s)
\end{split}
\end{equation}
Hence, combining \eqref{limiteuv1}, \eqref{limiteuv2}, \eqref{limiteuv3}, \eqref{limiteuv4}, \eqref{limiteuv5} and \eqref{limiteuv6}, we get
\[
\int_{\R^N\setminus \overline{\Omega}} C_\Omega (y) (v-u)(y) \varphi(y)\, dy = 0,
\]
namely $u=v$ a.e. on the support of $\varphi$. Since $\varphi$ is arbitrary, we obtain the claim in~\eqref{claimuv}.

{F}rom \eqref{claimuv} we infer that~$u_k$ converges to~$u$ a.e. in~$\R^N$. This and~\eqref{vdefn} give that~$u$ satisfies
\[
\int_{(0,1)}\Ns u\,d\mu(s)=0.
\]
This shows that~$u \in \widetilde{\mathcal H}_{0, \mu}(\Omega)$, as desired.
\end{proof}

For further reference, we recall the following result stated in~\cite[Lemma~3]{SV12}:

\begin{lemma}\label{SV12}
Assume that $f$ satisfies~\eqref{AR1} and~\eqref{AR2}.
Then, for any~$\varepsilon>0$ there exists~$\delta=\delta(\varepsilon)$ such that, for a.e.~$x\in \Omega$ and for any~$t\in \R$,
\begin{eqnarray}&&
|f(x,t)|\leq 2\varepsilon|t|+p\delta(\varepsilon)|t|^{p-1}\nonumber
\\
{\mbox{and }}&&
\label{F1}
|F(x,t)|\leq \varepsilon|t|^2+\delta(\varepsilon)|t|^{p},
\end{eqnarray}
where $F$ is defined as in~\eqref{definizioneF}.
\end{lemma}

\section{Existence of a Mountain Pass solution
and proof of Theorem~\ref{MPT}}\label{sec:mp}

In this section we provide the existence of a
Mountain Pass solution for problem~\eqref{problemasottocritico}. To this end, we start providing the definition of weak solution for problem~\eqref{problemasottocritico}.

\begin{definition}\label{weakdefn}
We say that $u\in \mathcal{H}_{\alpha, \mu}(\Omega)$ is a weak solution of problem~\eqref{problemasottocritico} if, for any~$v\in  \mathcal{H}_{\alpha, \mu}(\Omega)$,
\begin{equation}\label{weaksolution}
\begin{split}
\alpha \int_\Omega \nabla u\cdot\nabla v\, dx &+\int_{(0, 1)} \frac{c_{N, s}}{2} \iint_{\Q} \frac{(u(x)-u(y))(v(x)-v(y))}{|x-y|^{N+2s}} \,dx\,dy\, d\mu(s)\\
&+\int_\Omega uv\, dx = \lambda\int_\Omega uv\, dx +\int_\Omega f(x, u) v \, dx.
\end{split}
\end{equation}
\end{definition}

It is easy to check that if $f$ verifies~\eqref{AR1}, then the functional $I:\mathcal{H}_{\alpha, \mu}(\Omega)\to\R$ defined as
\begin{equation}\label{fun1}
\begin{split}
I(u):&= \frac{\alpha}{2}\int_\Omega |\nabla u|^2 \,dx +\int_{(0, 1)} \frac{c_{N, s}}{4} \iint_{\Q} \frac{|u(x)-u(y)|^2}{|x-y|^{N+2s}} \,dx\,dy\, d\mu(s)\\
&\qquad+\frac12 \int_{\Omega} |u|^2 \, dx -\frac{\lambda}{2} \int_{\Omega} |u|^2 \, dx - \int_\Omega F(x, u) \, dx
\end{split}
\end{equation}
is of class $C^1$ and satisfies, for any~$u$, $v\in \mathcal{H}_{\alpha, \mu}(\Omega)$,
\begin{equation}\label{diff1}
\begin{split}
\langle I'(u), v\rangle &=\alpha \int_\Omega \nabla u\cdot\nabla v\, dx +\int_{(0, 1)} \frac{c_{N, s}}{2} \iint_{\Q} \frac{(u(x)-u(y))(v(x)-v(y))}{|x-y|^{N+2s}} \,dx\,dy\, d\mu(s)\\
&\qquad+\int_\Omega uv\, dx - \lambda\int_\Omega uv\, dx - \int_\Omega f(x, u) v \, dx.
\end{split}
\end{equation}

Moreover, the critical points of $I$ are weak solutions of problem~\eqref{problemasottocritico} (see e.g.~\cite[Proposition~4.2]{TUTTI}).

We now take into account the case $\lambda<\lambda_1$ and we prove that sequences approaching a given energy level are bounded.  We point out that a similar result holds for the complementary case~$\lambda\ge\lambda_1$, but it requires a different care (see the forthcoming Proposition~\ref{boundedlinking}).

Both the results will be an useful ingredient to prove that the Palais-Smale condition is satisfied at any level.

\begin{proposition}\label{PropPS1}
Let $\lambda<\lambda_1$.
Assume that $f$ satisfies~\eqref{AR1}, \eqref{AR2} and~\eqref{AR3}. Let $c\in\R$ and let $u_n$ be a sequence in $\mathcal{H}_{\alpha, \mu}(\Omega)$ such that
\begin{equation}\label{bdd1}
\lim_{n\to +\infty} I(u_n)=c
\end{equation}
and
\begin{equation}\label{bdd2}
\lim_{n\to +\infty} \sup_{\substack{v\in \mathcal{H}_{\alpha, \mu}(\Omega)\\ \|v\|_{\alpha, \mu}=1}} |\langle I'(u_n), v\rangle| =0.
\end{equation}

Then, $u_n$ is bounded in $\mathcal{H}_{\alpha, \mu}(\Omega)$.
\end{proposition}

\begin{proof}
By~\eqref{bdd1} and~\eqref{bdd2}, there exists a positive constant $c_1$ such that, for\footnote{Notice that, without loss of generality, we can suppose that~$\|u_n\|_{\alpha, \mu}\neq0$ for any~$n\in\N$.} any~$n\in\N$,
\begin{equation*}
|I(u_n)|\le c_1\quad\mbox{ and }\quad \left|\left\langle I'(u_n), \frac{u_n}{\|u_n\|_{\alpha, \mu}}\right\rangle\right|\le c_1.
\end{equation*}
{F}rom this, if~$\vartheta$ is as in~\eqref{AR3}, we see that
\begin{equation}\label{c11}
\vartheta I(u_n) - \langle I'(u_n), u_n\rangle \le c_1\big(\vartheta +\|u_n\|_{\alpha, \mu}\big).
\end{equation}

Moreover, we take $r$ as in assumption~\eqref{AR3}. 
Then,
\begin{equation}\label{cnaoj}
\int_{\Omega \cap\{|u_n|> r\}}\Big(f(x, u_n) u_n -\vartheta F(x, u_n)\Big)\,dx \ge 0.
\end{equation}
On the other hand, if $r>0$, by using Lemma~\ref{SV12} with~$\varepsilon:=1$ we get that
\begin{equation}\label{ler}
\left|\;\int_{\Omega\cap\{|u_n|\le r\}}  \Big(f(x, u_n) u_n -\vartheta F(x, u_n)\Big) \,dx\right| \le \left(2r^2 +p\delta(1) r^p +\vartheta r^2 +\vartheta\delta(1) r^p\right)|\Omega|=:C_r.
\end{equation}

Now, by~\eqref{fun1}, \eqref{diff1}, \eqref{normaalfamu}, \eqref{cnaoj} and~\eqref{ler}, we conclude that
\begin{equation}\label{c12}
\begin{split}
\vartheta I(u_n) - \langle I'(u_n), u_n\rangle & =\alpha\left(\frac{\vartheta}{2} -1\right) \int_\Omega |\nabla u_n|^2 \,dx +\frac12 \left(\frac{\vartheta}{2} -1\right) \int_{(0, 1)} [u_n]^2_s \, d\mu(s)\\
&\qquad +\left(\frac{\vartheta}{2} -1\right) \int_\Omega |u_n|^2 \,dx -\lambda\left(\frac{\vartheta}{2} -1\right) \int_\Omega |u_n|^2 \,dx\\
&\qquad +\int_\Omega \Big(f(x, u_n)u_n -\vartheta F(x, u_n)\Big) \,dx\\
&\ge \left(\frac{\vartheta}{2} -1\right) \|u_n\|^2_{\alpha, \mu} -\lambda\left(\frac{\vartheta}{2} -1\right)  \|u_n\|^2_{\alpha, \mu}\\
&\qquad-\left|\;\int_{\Omega\cap\{|u_n|\le r\}}  \Big(f(x, u_n) u_n -\vartheta F(x, u_n)\Big) \,dx\right| \\
&\ge(1-\lambda)\left(\frac{\vartheta}{2} -1\right) \|u_n\|^2_{\alpha, \mu} -C_r.
\end{split}
\end{equation}
Accordingly, by~\eqref{c11} and~\eqref{c12}, since~$\lambda<\lambda_1=1$ (recall~\eqref{primouguale1}) and~$\vartheta>2$, we infer that
\begin{equation}\label{diverso}
\|u_n\|^2_{\alpha, \mu} \le \frac{2 c_1(\vartheta +\|u_n\|_{\alpha, \mu}) +2C_r} {(1-\lambda)(\vartheta-2)},
\end{equation}
namely $u_n$ is bounded in $\mathcal{H}_{\alpha, \mu}(\Omega)$. This concludes the proof.
\end{proof}

We now show that any bounded sequence in~$\mathcal{H}_{\alpha, \mu}(\Omega)$ converges strongly to a function in~$\mathcal{H}_{\alpha, \mu}(\Omega)$. It is worth noting that, in this case, our result does not require the distinction between~$\lambda<\lambda_1$ and~$\lambda\ge\lambda_1$.

\begin{proposition}\label{PropPS2}
Let $\lambda\in\R$ and assume that $f$ satisfies~\eqref{AR1}. Let $u_n$ be a bounded sequence in~$\mathcal{H}_{\alpha, \mu}(\Omega)$ satisfying~\eqref{bdd2}.

Then, there exists $u\in\mathcal{H}_{\alpha, \mu}(\Omega)$ such that
\[
{\mbox{$u_n\to u$ in~$\mathcal{H}_{\alpha, \mu}(\Omega)$ as~$n\to+\infty$.}}
\]
\end{proposition}

\begin{proof}
Since $u_n$ is bounded in $ \mathcal{H}_{\alpha, \mu}(\Omega)$
(and~$ \mathcal{H}_{\alpha, \mu}(\Omega)$ is a Hilbert space, thanks
to~\cite[Corollary~2.2]{TUTTI}), by Proposition~\ref{embeddings} there exists $u\in\mathcal{H}_{\alpha, \mu}(\Omega)$ such that, up to subsequences, as~$n\to+\infty$,
\begin{align}\label{weakconv}
&u_n\rightharpoonup u \quad\mbox{ in }  \mathcal{H}_{\alpha, \mu}(\Omega),\\ \label{convlp}
{\mbox{and }} \quad &u_n\to u \quad\mbox{ in }  L^p(\Omega), \quad\mbox{ for any } p\in [1, 2^*_{s_\sharp}).
\end{align}

Now, by~\eqref{weakconv} and~\eqref{bdd2}, we have
\[
\lim_{n\to +\infty}\langle I'(u_n), u_n -u\rangle = 0,
\]
namely
\[
\begin{split}
\lim_{n\to +\infty}\Bigg(&\alpha \int_\Omega \nabla u_n\cdot\nabla (u_n-u) \, dx\\
&\quad+\int_{(0, 1)} \frac{c_{N, s}}{2} \iint_{\Q} \frac{(u_n(x)-u_n(y))\big((u_n-u)(x)-(u_n-u)(y)\big)}{|x-y|^{N+2s}} \,dx\,dy\, d\mu(s)\\
&\quad+(1-\lambda)\int_\Omega u_n(u_n-u)\, dx -\int_\Omega f(x, u_n) (u_n-u) \, dx\Bigg)=0.
\end{split}
\]

We observe that, by hypothesis~\eqref{AR1}, 
\[
\left|\int_\Omega f(x, u_n) (u_n-u) \, dx\right|\le \int_\Omega \big(
a_1 +a_2 |u_n|^{p-1}\big) |u_n - u| \, dx.
\]
Let $p'$ denote the conjugate exponent of $p$. Now, since $a_1 +a_2 |u_n|^{p-1} \in L^{p'}(\Omega)$, by~\eqref{convlp} we infer that
\[
\lim_{n\to +\infty} \int_\Omega f(x, u_n) (u_n-u) \, dx =0.
\]
Thus, we have that
\begin{equation}\label{lim1}
\begin{split}
\lim_{n\to +\infty}\Bigg(&\alpha \int_\Omega \nabla u_n\cdot\nabla (u_n-u) \, dx\\
&\quad+\int_{(0, 1)} \frac{c_{N, s}}{2} \iint_{\Q} \frac{\big(u_n(x)-u_n(y))((u_n-u)(x)-(u_n-u)(y)\big)}{|x-y|^{N+2s}} \,dx\,dy\, d\mu(s)\\
&\quad+(1-\lambda)\int_\Omega u_n(u_n-u)\, dx\Bigg)=0.
\end{split}
\end{equation}

On the other hand, by~\eqref{weakconv} and~\eqref{convlp},
it follows that
\begin{equation}\label{lim2}
\begin{split}
\lim_{n\to +\infty}\Bigg(&\alpha \int_\Omega \nabla u\cdot\nabla (u_n-u) \, dx\\
&\quad+\int_{(0, 1)} \frac{c_{N, s}}{2} \iint_{\Q} \frac{\big(u(x)-u(y))((u_n-u)(x)-(u_n-u)(y)\big)}{|x-y|^{N+2s}} \,dx\,dy\, d\mu(s)\\
&\quad+(1-\lambda)\int_\Omega u(u_n-u)\, dx\Bigg)=0.
\end{split}
\end{equation}
Combining~\eqref{lim1} and~\eqref{lim2}, we get
\[
\lim_{n\to +\infty}\Big(\|u_n -u\|^2_{\alpha, \mu} -\lambda\|u_n -u\|^2_{L^2(\Omega)}\Big)=0.
\]
Thus, by using~\eqref{convlp}, we infer that
\[
\lim_{n\to +\infty}\|u_n -u\|^2_{\alpha, \mu} =0,
\]
as desired.
\end{proof}

It remains to show that, if $\lambda <\lambda_1$, the functional~$I$ in~\eqref{fun1} fulfills the geometry of the Mountain Pass Theorem.

\begin{proposition}\label{geometriaMPAR1}
Let $\lambda<\lambda_1$.
Assume that $f$ satisfies assumptions~\eqref{AR1} and~\eqref{AR2}.

Then, there exist~$\rho>0$ and~$\beta>0$ such that, for any~$u\in \mathcal{H}_{\alpha,\mu}(\Omega)$ with~$\|u\|_{\alpha,\mu}=\rho$,
it holds that~$I(u)\geq \beta$.
\end{proposition}

\begin{proof}
Let $u\in \mathcal{H}_{\alpha,\mu}(\Omega)$. By~\eqref{fun1}, \eqref{F1}, \eqref{normaalfamu} and Proposition~\ref{embeddings}, we infer that for any $\varepsilon>0$,
\begin{equation}\label{bvcxzaer2tyw}
\begin{split}
I(u) & \ge \frac12 \|u\|^2_{\alpha, \mu} -\frac{\lambda}{2} \|u\|^2_{L^2(\Omega)} -\varepsilon\|u\|^2_{L^2(\Omega)} - \delta(\varepsilon)\|u\|^p_{L^p(\Omega)}\\
& \ge\left(\frac{1-\lambda}{2} -\varepsilon\right) \|u\|^2_{\alpha, \mu} -\delta(\varepsilon)\|u\|^p_{L^p(\Omega)}\\
&\ge \|u\|^2_{\alpha, \mu} \left(\frac{1-\lambda}{2}-\varepsilon -\delta(\varepsilon) c  \|u\|^{p-2}_{\alpha, \mu}\right),
\end{split}
\end{equation}
where $c=c(N, \Omega, s_\sharp)>0$.

Now, since~$\lambda<\lambda_1=1$ (recall~\eqref{primouguale1}), we can take 
\[
\varepsilon\in\left(0, \frac{1-\lambda}{2}\right).
\]
Moreover, we recall that~$p\in (2, 2^*_{s_\sharp})$ (thanks to the
assumption in~\eqref{AR1}) and so we take
\[
\rho \in\left(0, \left(\frac{1-\lambda-2\varepsilon}{2\delta(\varepsilon) c} \right)^{\frac{1}{p-2}}\right).
\]
In this way,
\[
\frac{1-\lambda}{2}-\varepsilon -\delta(\varepsilon) c\rho^{p-2}>0, 
\]
and therefore,
plugging this information into~\eqref{bvcxzaer2tyw}, we gather that
\[
\inf_{\substack{u\in \mathcal{H}_{\alpha,\mu}(\Omega)\\ \|u\|_{\alpha, \mu}=\rho}} I(u)\ge \rho^2 \left(\frac{1-\lambda}{2}-\varepsilon -\delta(\varepsilon) c\rho^{p-2}\right) =:\beta >0,
\]
as desired.
\end{proof}

\begin{proposition}\label{geometriaMPAR2}
Let $\lambda<\lambda_1$ and assume that $f$ satisfies~\eqref{AR1}, \eqref{AR2} and~\eqref{AR4}.

Then, there exists $e\in \mathcal{H}_{\alpha,\mu}(\Omega)$ such that~$e\geq 0$ a.e. in~$\R^N$, $\|e\|_{\alpha,\mu}>\rho$ and~$I(e)<\beta$, where~$\rho$ and~$\beta$ are given in Proposition~\ref{geometriaMPAR1}.
\end{proposition}

\begin{proof}
We fix $u_0\in \mathcal{H}_{\alpha,\mu}(\Omega)$ such that $\|u_0\|_{\alpha,\mu}=1$ and $u_0\geq 0$ a.e. in $\R^N$. Moreover, let $t>0$. We point out that
\begin{alignat*}{2}
&\mbox{either $ \ -\displaystyle\frac{\lambda t^2}{2}\|u_0\|_{L^2(\Omega)}^2\le 0$} &&\mbox{ \ if $\lambda\in [0,\lambda_1)$,}\\
&\mbox{or \qquad$-\displaystyle\frac{\lambda t^2}{2}\|u_0\|_{L^2(\Omega)}^2\le -\dfrac{\lambda t^2}{2}\|u_0\|_{\alpha,\mu}^2$} &&\mbox{ \ if $\lambda<0$.}
\end{alignat*}
In light of this, \eqref{fun1}, \eqref{normaalfamu} and~\eqref{AR4}, we infer that
\[
\begin{aligned}
I(tu_0)&=\frac{t^2}{2}\|u_0\|_{\alpha,\mu}^2-\frac{\lambda t^2}{2}\|u_0\|_{L^2(\Omega)}^2
-\int_\Omega F(x,tu_0)\,dx  \\
&\leq \frac{t^2}{2} \big(1+\max\{0,-\lambda\}\big)\|u_0\|_{\alpha,\mu}^2
-a_3t^{\widetilde{\vartheta}}\|u_0\|_{L^{\widetilde{\vartheta}}(\Omega)}^{\widetilde{\vartheta}}
+\int_\Omega a_4(x)\,dx.
\end{aligned}
\]
Since $\widetilde{\vartheta}>2$, passing to the limit we get
\[
\lim_{t\to+\infty}I(tu_0)=-\infty.
\]
Therefore, taking $e:=tu_0$ for $t$ large enough we have the desired result.
\end{proof}

We are now in the position of completing the proof of Theorem~\ref{MPT}.

\begin{proof}[Proof of Theorem~\ref{MPT}]
When $\lambda<\lambda_1$, the validity of the Palais-Smale condition at any level~$c\in\R$ is provided by Propositions~\ref{PropPS1} and~\ref{PropPS2}. On the other hand, from Propositions~\ref{geometriaMPAR1} and~\ref{geometriaMPAR2} we deduce that the functional~$I$ satisfies the geometric properties of the Mountain Pass Theorem.

Thus, we can apply the Mountain Pass Theorem (see e.g.~\cite[Theorem~6.1]{struwe}) and conclude that there exists a critical point $u\in\mathcal{H}_{\alpha,\mu}(\Omega)$ of $I$ such that
\[
I(u)\geq \beta >0=I(0).
\]
Hence, in particular, $u\neq 0$. This concludes the proof.
\end{proof}

\section{Existence of a Linking solution and proof of Theorem~\ref{LINK}}\label{sec-link}

In this section we address the study of problem~\eqref{problemasottocritico} when $\lambda\ge\lambda_1$ and we use the Linking Theorem to prove the existence of the nontrivial solution claimed in Theorem~\ref{LINK}.

We notice that, since $\lambda\ge \lambda_1$, we can assume that~$\lambda$ falls between two consecutive eigenvalues of the operator~$L_{\alpha,\mu}+Id$, namely
\begin{equation}\label{successivi}
\lambda\in [\lambda_i, \lambda_{i+1})\quad\mbox{ for some } i\in\N\setminus\{0\}.
\end{equation}

In addition, we point out that to develop the Linking argument we need to cast our problem in a suitable space which admits a direct sum decomposition. In~\cite[Theorem~1.5]{TUTTI}, it has been shown that the sequence~$e_k$ of eigenfunctions of the operator~$L_{\alpha,\mu}$
determines a complete orthogonal system in $L^2(\Omega)$. However, since~$e_k$ is an orthogonal system in $\mathcal{H}_{\alpha, \mu}(\Omega)$ but the completeness is not guaranteed,  our choice cannot fall on the space~$\mathcal{H}_{\alpha, \mu}(\Omega)$. For this reason, we restrict ourselves to the subspace~$\widetilde{\mathcal{H}}_{\alpha, \mu}(\Omega)$ introduced in~\eqref{tildeacca} (recall, indeed, that
thanks to Proposition~\ref{sistemacompleto}
the sequence~$e_k$ provides a complete orthogonal system in~$\widetilde{\mathcal H}_{0, \mu}(\Omega)$).

Due to this choice, we rephrase the definition of weak solution for problem~\eqref{problemasottocritico} as follows:

\begin{definition}\label{wdlinking}
We say that $u\in \widetilde{\mathcal{H}}_{\alpha, \mu}(\Omega)$ is a weak solution of problem~\eqref{problemasottocritico}
if, for any~$v\in\widetilde{\mathcal{H}}_{\alpha, \mu}(\Omega)$,
\begin{equation}\label{wslinking}
\begin{split}
\alpha \int_\Omega \nabla u\cdot\nabla v\, dx &+\int_{(0, 1)} \frac{c_{N, s}}{2} \iint_{\Q} \frac{(u(x)-u(y))(v(x)-v(y))}{|x-y|^{N+2s}} \,dx\,dy\, d\mu(s)\\
&+\int_\Omega uv\, dx = \lambda\int_\Omega uv\, dx+ \int_\Omega f(x, u) v \, dx.
\end{split}
\end{equation}
\end{definition}

Moreover, if $f$ satisfies assumption~\eqref{AR1}, then the functional $\mathcal I:\widetilde{\mathcal{H}}_{\alpha, \mu}(\Omega)\to\R$ defined as
\begin{equation}\label{funlink}
\begin{split}
\mathcal I(u):&= \frac{\alpha}{2}\int_\Omega |\nabla u|^2 \,dx +\int_{(0, 1)} \frac{c_{N, s}}{4} \iint_{\Q} \frac{|u(x)-u(y)|^2}{|x-y|^{N+2s}} \,dx\,dy\, d\mu(s)\\
&\qquad+\frac12 \int_{\Omega} |u|^2 \, dx -\frac{\lambda}{2} \int_{\Omega} |u|^2 \, dx - \int_\Omega F(x, u) \, dx
\end{split}
\end{equation}
is of class $C^1$ and satisfies, for any~$u$, $v\in \widetilde{\mathcal{H}}_{\alpha, \mu}(\Omega)$,
\begin{equation}\label{diff2}
\begin{split}
\langle \mathcal I'(u), v\rangle &=\alpha \int_\Omega \nabla u\cdot\nabla v\, dx +\int_{(0, 1)} \frac{c_{N, s}}{2} \iint_{\Q} \frac{(u(x)-u(y))(v(x)-v(y))}{|x-y|^{N+2s}} \,dx\,dy\, d\mu(s)\\
&\qquad+\int_\Omega uv\, dx - \lambda\int_\Omega uv\, dx -\int_\Omega f(x, u) v \, dx.
\end{split}
\end{equation}

Moreover, the critical points of $\mathcal I$ are weak solutions of problem~\eqref{problemasottocritico} (see e.g.~\cite[Proposition~4.2]{TUTTI}).

We now show that, if~\eqref{successivi} holds true and~$u_n$ approaches a suitable energy level, then~$u_n$ is bounded.  The wording of the next result closely resembles that of Proposition~\ref{PropPS1},  but some technical difficulties arise in the proof. To be more precise, in light of the assumption~\eqref{successivi},  we have that inequality~\eqref{diverso} can not be deduced by~\eqref{c12}.
This requires a careful approach to this case.

\begin{proposition}\label{boundedlinking}
Let $\lambda$ be as in~\eqref{successivi} and suppose that~$f$ satisfies~\eqref{AR1}, \eqref{AR2}, \eqref{AR3} and~\eqref{AR4}. Let~$c\in\R$ and let~$u_n$ be a sequence in $\widetilde{\mathcal{H}}_{\alpha, \mu}(\Omega)$ such that
\begin{equation}\label{bdd12}
\lim_{n\to +\infty} \mathcal I(u_n)=c
\end{equation}
and
\begin{equation}\label{bdd22}
\lim_{n\to +\infty} \sup_{\substack{v\in \widetilde{\mathcal{H}}_{\alpha, \mu}(\Omega)\\ \|v\|_{\alpha, \mu}=1}} |\langle \mathcal I'(u_n), v\rangle| =0.
\end{equation}

Then, $u_n$ is bounded in~$\widetilde{\mathcal{H}}_{\alpha, \mu}(\Omega)$.
\end{proposition}

\begin{proof}
Let $k\in (2, \vartheta)$. By~\eqref{bdd12} and~\eqref{bdd22} we infer that there exists a constant $c_1>0$ such that for any $n\in\N$,
\begin{equation}\label{ndkjsc}
k \mathcal I(u_n) -\langle \mathcal I'(u_n), u_n\rangle \le c_1 (k +\|u_n\|_{\alpha, \mu}).
\end{equation}
Let $r$ be as in assumption~\eqref{AR3} and notice that
\begin{equation}\label{lnsdnv}
\int_{\Omega\cap\{|u_n|> r\}} \Big(f(x, u_n) u_n -k F(x, u_n)\Big)\,dx >0.
\end{equation}
Moreover, if $r>0$, then by using Lemma~\ref{SV12} with~$\varepsilon:=1$ we get that
\begin{equation}\label{ler2}
\left|\;\int_{\Omega\cap\{|u_n|\le r\}}  \big(f(x, u_n) u_n -\vartheta F(x, u_n)\big) \,dx\right| \le \left(2r^2 +p\delta(1) r^p + \vartheta r^2 +\vartheta\delta(1) r^p\right)|\Omega|=:C_r.
\end{equation}
Hence, by~\eqref{funlink}, \eqref{diff2},
\eqref{lnsdnv} and~\eqref{ler2}, we infer that
\[
\begin{split}
k\, \mathcal I(u_n) -\langle \mathcal I'(u_n), u_n\rangle &\ge\left(\frac{k}{2}-1\right)\|u_n\|^2_{\alpha, \mu}-\lambda\left(\frac{k}{2}-1\right)\|u_n\|^2_{L^2(\Omega)} \\
&\qquad+\int_\Omega \big(f(x, u_n) u_n -k F(x, u_n)\big) \,dx\\
& \ge\left(\frac{k}{2}-1\right)\|u_n\|^2_{\alpha, \mu}-\lambda\left(\frac{k}{2}-1\right)\|u_n\|^2_{L^2(\Omega)}\\
&\qquad - \left|\;\int_{\Omega\cap\{|u_n|\le r\}}  \Big(f(x, u_n) u_n -\vartheta F(x, u_n)\Big) \,dx\right| +(\vartheta -k)\int_\Omega F(x, u_n) \,dx\\
&\ge\left(\frac{k}{2}-1\right)\|u_n\|^2_{\alpha, \mu}-\lambda\left(\frac{k}{2}-1\right)\|u_n\|^2_{L^2(\Omega)}\\
&\qquad +(\vartheta -k)\int_\Omega F(x, u_n) \,dx- C_r.
\end{split}
\]

Let now~$\widetilde\vartheta$ be as in hypothesis~\eqref{AR4}. We notice that, by using both H\"older and Young inequalities,
for any~$\varepsilon>0$ and for any~$u\in \widetilde{\mathcal{H}}_{\alpha, \mu}(\Omega)$,
\[
\|u\|^2_{L^2(\Omega)}\le \frac{2\,\varepsilon^{\widetilde\vartheta/2}}{\widetilde\vartheta} \|u\|^{\widetilde\vartheta}_{L^{\widetilde\vartheta}(\Omega)} +\frac{\widetilde\vartheta -2}{\widetilde\vartheta\,\varepsilon^{\widetilde\vartheta/{(\widetilde\vartheta -2)}}} |\Omega|.
\]
In light of this, recalling that $\vartheta>k$ and that~\eqref{AR4} holds, we infer that
\[
\begin{split}
k\, \mathcal I(u_n) -\langle \mathcal I'(u_n), u_n\rangle&\ge\left(\frac{k}{2}-1\right)\|u_n\|^2_{\alpha, \mu}-\lambda\left(\frac{k}{2}-1\right)\|u_n\|^2_{L^2(\Omega)}\\
&\qquad +(\vartheta -k) a_3 \|u_n\|^{\widetilde\vartheta}_{L^{\widetilde\vartheta}(\Omega)} +(\vartheta - k)\int_\Omega a_4(x) \,dx -C_r\\
&\ge\left(\frac{k}{2}-1\right)\|u_n\|^2_{\alpha, \mu} + \left((\vartheta-k)a_3-\frac{2\,\lambda}{\widetilde\vartheta}\,\varepsilon^{\widetilde\vartheta/2}\left(\frac{k}{2}-1 \right)\right)\|u_n\|^{\widetilde\vartheta}_{L^{\widetilde\vartheta}(\Omega)} -c_2,
\end{split}
\]
where
\[
c_2:= C_r +\frac{|\Omega|(\widetilde\vartheta-2)}{\widetilde\vartheta\,\varepsilon^{\widetilde\vartheta/(\widetilde\vartheta-2)}} +(\vartheta - k)\int_\Omega a_4(x) \,dx. 
\]

Now, by taking
\[
\varepsilon\in\left(0, \left(\frac{a_3\widetilde\vartheta (\vartheta-k)}{\lambda(k-2)}\right)^{2/\widetilde\vartheta}\right),
\]
we get
\[
k\, \mathcal I(u_n) -\langle \mathcal I'(u_n), u_n\rangle\ge \left(\frac{k}{2}-1\right)\|u_n\|^2_{\alpha, \mu} -c_2.
\]
This, together with~\eqref{ndkjsc}, entails that 
\[
\|u_n\|^2_{\alpha, \mu}\le \frac{2c_1 (k +\|u_n\|_{\alpha, \mu}) +2c_2}{k-2},
\]
namely $u_n$ is bounded in $\widetilde{\mathcal{H}}_{\alpha, \mu}(\Omega)$, as desired.
\end{proof}

The next result states that a bounded sequence in~$\widetilde{\mathcal{H}}_{\alpha, \mu}(\Omega)$ converges strongly
to a function in~$\widetilde{\mathcal{H}}_{\alpha, \mu}(\Omega)$.
We point out that this result holds for any value of~$\lambda$, without
distinguishing between the cases~$\lambda<\lambda_1$ and~$\lambda\ge\lambda_1$.

\begin{proposition}\label{PropPS22}
Let $\lambda\in\R$ and assume that $f$ satisfies~\eqref{AR1}. Let~$u_n$ be a bounded sequence in~$\widetilde{\mathcal{H}}_{\alpha, \mu}(\Omega)$ satisfying~\eqref{bdd22}.

Then, there exists $u\in\widetilde{\mathcal{H}}_{\alpha, \mu}(\Omega)$ such that
\[
{\mbox{$u_n\to u$ in~$\widetilde{\mathcal{H}}_{\alpha, \mu}(\Omega)$
as~$n\to+\infty$.}}
\]
\end{proposition}

The proof of Proposition~\ref{PropPS22} follows the same lines as the
one of Proposition~\ref{PropPS2}, relying now also on Proposition~\ref{chiusoconvergenzadebole} to use the fact that the space~$\widetilde{\mathcal{H}}_{\alpha, \mu}(\Omega)$
is closed with respect to the weak convergence.

We now prove that the geometry of the Linking Theorem is satisfied.
To this aim, we observe that the space $\widetilde{\mathcal{H}}_{\alpha, \mu}(\Omega)$ admits a direct sum decomposition. Indeed, by the definition in~\eqref{tildeacca} (and Proposition~\ref{sistemacompleto}
when~$\alpha=0$) there exists a complete orthogonal system~$e_k$ for the space~$\widetilde{\mathcal{H}}_{\alpha, \mu}(\Omega)$.
Hence, for all~$i\in\N\setminus\{0\}$, we introduce the spaces
\begin{equation}\label{directsum}
H_i:= \mbox{span}\{e_1,\dots, e_i\}\quad\mbox{ and }\quad H^{\perp}_i:= \overline{\mbox{span}\{e_{i+1},\dots\}}.
\end{equation}
In order to not weigh down the notation,  in the previous writing we have not highlighted explicitly the dependence on~$\alpha$ and~$\mu$. However, it is clear that~$H_i=H_{i, \alpha, \mu}$ and~$H^\perp_i=H^\perp_{i, \alpha, \mu}$.

Thus, by~\eqref{directsum} we infer the following direct sum decomposition
\[
\widetilde{\mathcal{H}}_{\alpha, \mu}(\Omega) = H_i \oplus  H^{\perp}_i.
\]
\begin{remark}
We notice that by~\eqref{directsum}, it follows that
\begin{equation}\label{poincHi}
\|u\|_{\alpha,\mu}^2\leq \lambda_i\|u\|_{L^2(\Omega)}^2
\quad \mbox{for any } u\in H_i,
\end{equation}
and
\begin{equation}\label{poincHiperp}
\|u\|_{\alpha,\mu}^2\geq \lambda_{i+1}\|u\|_{L^2(\Omega)}^2
\quad \mbox{for any } u\in H_i^\perp.
\end{equation}

Indeed, if $u\in H_i$, then 
\[
u=\sum_{k=1}^ia_ke_k.
\]
Hence, by construction,
\[
\|u\|_{\alpha,\mu}^2=
\sum_{k=1}^ia_k^2\|e_k\|_{\alpha,\mu}^2
=\sum_{k=1}^ia_k^2\lambda_k\|e_k\|_{L^2(\Omega)}^2 \leq \lambda_i\sum_{k=1}^ia_k^2\|e_k\|_{L^2(\Omega)}^2
=\lambda_i\|u\|_{L^2(\Omega)}^2,
\]
namely~\eqref{poincHi} is satisfied.

In a similar fashion, if $u\in H_i^\perp$, then
\[
u=\sum_{k=i+1}^{+\infty}a_ke_k.
\]
Thus,
\[
\|u\|_{\alpha,\mu}^2=
\sum_{k=i+1}^{+\infty}a_k^2\|e_k\|_{\alpha,\mu}^2
=\sum_{k=i+1}^{+\infty}a_k^2\lambda_k\|e_k\|_{L^2(\Omega)}^2 
\geq \lambda_{i+1}\sum_{k=i+1}^{+\infty}a_k^2\|e_k\|_{L^2(\Omega)}^2
=\lambda_{i+1}\|u\|_{L^2(\Omega)}^2,
\]
which proves~\eqref{poincHiperp}.
\end{remark}

We now focus on showing that the geometry of the Linking Theorem is satisfied.

\begin{proposition}\label{LGEO1}
Let~$\lambda$ as in~\eqref{successivi} and suppose that~$f$ satisfies~\eqref{AR1} and~\eqref{AR2}.

Then, there exist $\rho>0$ and $\beta>0$ such that, for any~$u\in H_i^\perp$ with~$\|u\|_{\alpha,\mu}=\rho$,
it holds that~$\mathcal{I}(u)\geq \beta$.
\end{proposition}

\begin{proof}
Let $u\in H_i^\perp$. By~\eqref{F1}, \eqref{funlink}, \eqref{poincHiperp}, \eqref{normaalfamu} and Proposition~\ref{embeddings}, we infer that, for any~$\varepsilon>0$,
\[
\begin{aligned}
\mathcal{I}(u)&=
\frac{1}{2}\|u\|_{{\alpha,\mu}}^2-\frac{\lambda}{2}\|u\|_{L^2(\Omega)}^2-\int_\Omega F(x,u)\,dx \\
&\geq \frac{1}{2}\|u\|_{{\alpha,\mu}}^2-\frac{\lambda}{2\lambda_{i+1}}\|u\|_{\alpha,\mu}^2-\varepsilon \|u\|_{L^2(\Omega)}^2-\delta(\varepsilon)\|u\|_{L^p(\Omega)}^p \\
&\geq \|u\|_{\alpha,\mu}^2\left(\frac{1}{2}- \frac{\lambda}{2\lambda_{i+1}}-\varepsilon-c\,\delta(\varepsilon) \|u\|_{\alpha,\mu}^{p-2}\right),
\end{aligned}
\]
being $c=c(N,\Omega,s_\sharp)>0$.

Now, let
\[
\varepsilon \in \left(0, \frac{1}{2}-\frac{\lambda}{2\lambda_{i+1}} \right)
\qquad{\mbox{and}}\qquad
\rho \in\left(0, \left(\frac{1}{2c\,\delta(\varepsilon)}\left(1-\frac{\lambda}{\lambda_{i+1}}\right) -\frac{\varepsilon}{2c\,\delta(\varepsilon)}\right)^{1/(p-2)}\right).
\]
In this way,
\[
\frac{1}{2}- \frac{\lambda}{2\lambda_{i+1}}-\varepsilon
-c\,\delta(\varepsilon) \rho^{p-2}>0
\]
and therefore
\[
\inf_{\substack{u\in H_i^{\perp} \\ \|u\|_{\alpha, \mu}=\rho}} \mathcal{I}(u)
\geq \rho^2\left(
\frac{1}{2}- \frac{\lambda}{2\lambda_{i+1}}-\varepsilon
-c\,\delta(\varepsilon)\rho^{p-2}
\right)=:\beta>0,
\]
as desired.
\end{proof}

It remains to show that there exists a suitable set in which the functional in~\eqref{funlink} is nonpositive.  In the next result we prove that such a set exists and it links with the set\footnote{For every subspace~$X$ of~$\mathcal H_{\alpha, \mu}(\Omega)$, the notation~$\partial_X$ stands for the boundary in the subspace~$X$.} $\partial_{H_i^\perp}(B_\rho\cap H_i^\perp)$, being~$\rho$ as in Proposition~\ref{LGEO1}.

\begin{proposition}\label{LGEO2}
Let $\lambda$ be as in~\eqref{successivi} and assume that~$f$ satisfies~\eqref{AR3}, \eqref{AR4} and~\eqref{AR5}. Let~$\rho>0$ be
as in Proposition~\ref{LGEO1}.

Then, there exist $R>\rho$ and~$e\in H_i^\perp$ such that the set
\[
\Sigma:=\partial_{H_i\oplus e} \Big((\overline{B_R}\cap H_i) \oplus \{te : t\in [0, R]\} \Big)
\]
satisfies~$\mathcal I (\Sigma)\le 0$.
\end{proposition}

\begin{proof}
We claim that
\begin{equation}\label{I<0}
\mathcal I(u)\le 0 \quad\mbox{ for any } u\in H_i.
\end{equation}
Indeed, by~\eqref{funlink}, \eqref{poincHi}, \eqref{AR3}, \eqref{AR5} and~\eqref{successivi}, we infer that
\[
\mathcal I(u) = \frac{1}{2}\|u\|_{{\alpha,\mu}}^2-\frac{\lambda}{2}\|u\|_{L^2(\Omega)}^2-\int_\Omega F(x,u)\,dx 
\le \frac12 (\lambda_i -\lambda) \|u\|_{L^2(\Omega)}^2\le 0,
\]
which establishes~\eqref{I<0}.

Moreover, let $u\in H_i$ and~$t>0$ and consider~$u+te_{i+1}$. Notice that~$e_{i+1}\in H_i^\perp$.
Now, up to a normalization, we can say that~$\|e_{i+1}\|_{L^2(\Omega)}^2=1$. Hence, 
\[
\|e_{i+1}\|_{\alpha, \mu}^2=\lambda_{i+1}\|e_{i+1}\|^2_{L^2(\Omega)} =\lambda_{i+1}.
\]
Accordingly, by~\eqref{funlink}, \eqref{AR4}, \eqref{successivi}, \eqref{directsum} and~\eqref{poincHi} we have 
\[
\begin{split}
\mathcal I(u+te_{i+1})&= \frac{1}{2}\|u\|_{{\alpha,\mu}}^2-\frac{\lambda}{2}\|u\|_{L^2(\Omega)}^2 +\frac{t^2}{2} \|e_{i+1}\|_{{\alpha,\mu}}^2 -\frac{\lambda t^2}{2} \|e_{i+1}\|_{L^2(\Omega)}^2 -\int_\Omega F(x,u+te_{i+1})\,dx \\
&\le \frac12 (\lambda_i -\lambda) \|u\|_{L^2(\Omega)}^2 +\frac{t^2}{2} (\lambda_{i+1}-\lambda) - a_3\, t^{\widetilde\vartheta} \int_\Omega \left|\frac{u}{t} +e_{i+1} \right|^{\widetilde\vartheta} \, dx -\|a_4\|_{L^1(\Omega)}\\
&\le \frac{t^2}{2} (\lambda_{i+1}-\lambda) - a_3\, t^{\widetilde\vartheta} \int_\Omega \left|\frac{u}{t} +e_{i+1} \right|^{\widetilde\vartheta} \, dx -\|a_4\|_{L^1(\Omega)}.
\end{split}
\]
Thus, being $\widetilde\vartheta>2$, we get 
\[
\lim_{t\to +\infty} \mathcal I(u+te_{i+1}) = -\infty.
\]
{F}rom this and~\eqref{I<0}, we can take $e:=e_{i+1}$ and $R$ large enough to get the desired result.
This concludes the proof.
\end{proof}

\begin{proof}[Proof of the Theorem~\ref{LINK}]
Since $\lambda\ge \lambda_1$, we can take $\lambda$ as in~\eqref{successivi}. {F}rom Popositions~\ref{boundedlinking} and~\ref{PropPS22},
we infer that the functional~$\mathcal I$ satisfies the Palais--Smale condition at any level $c\in\R$.

Moreover, by Propositions~\ref{LGEO1} and~\ref{LGEO2}, the geometry of the Linking Theorem is satisfied as well. 

Hence, we can apply the Linking Theorem (see e.g.~\cite[Theorem~5.3]{MR845785}). {F}rom this, we infer that there exists a critical point $u\in\widetilde{\mathcal{H}}_{\alpha, \mu}(\Omega)$ such that
\[
\mathcal I(u)\ge\beta>0 = \mathcal I(0),
\]
being $\beta$ as in Proposition~\ref{LGEO1}. In particular, $u\not\equiv 0$. This concludes the proof.
\end{proof}

\section{Examples and applications}\label{sec:appex}

In this section we illustrate some examples that arise from particular choices of the measure~$\mu$.
Indeed, the operator~$L_{\alpha, \mu}$ introduced in~\eqref{ELLE} has a broad formulation which allows us to produce a wide number of new interesting existence results for subcritical problems with Neumann boundary conditions. We exhibit some of these cases here below.

Throughout this section we assume that $f:\Omega\times \R \to \R$ is a Carath\'eodory function satisfying~\eqref{AR1}, \eqref{AR2},  \eqref{AR3}, \eqref{AR4} and~\eqref{AR5}.  Moreover, we will use the notation $H^s(\Omega)$ to identify the space
\[
H^s(\Omega):=\Big\{u:\R^N\to\R \mbox{ measurable } : \|u\|_{L^2(\Omega)} + [u]_s <+\infty\Big\},
\]
see e.g.~\cite[Section~3]{MR3651008}. Also, in the forthcoming results, we will use the notation~$
H^1(\Omega)\cap H^s (\Omega)$
to denote the space of functions~$u\in H^s (\Omega)$ whose restriction over~$\Omega$ belongs to~$H^1(\Omega)$.

The first existence result that we present here concerns with the mixed operator $-\alpha\Delta u + \beta (-\Delta)^s u$. 

\begin{corollary}\label{cor1}
Let $s\in (0, 1)$ and $\beta>0$.  Denote by $\widetilde\lambda_k$ the sequence of the eigenvalues of the operator
\[
-\alpha\Delta u + \beta (-\Delta)^s u
\]
under homogeneous Neumann conditions\footnote{We recall the notation introduced in~\eqref{eigennotation}}.

Then, if~$\alpha>0$, for any $\lambda\in\R$, the problem
\begin{equation}\label{P11}
\begin{cases}
-\alpha\Delta u + \beta (-\Delta)^s u +u=\lambda u +f(x,u)  &\mbox{in }\Omega,\\
\partial_\nu u(x)=0  &\mbox{for all } x\in \partial\Omega,\\
 \Ns u(x)=0 &\mbox{for all }x\in \R^N \setminus \overline{\Omega}, 
\end{cases}
\end{equation}
admits a nontrivial solution $u$ such that 
\begin{equation}\label{spazicor1}
\begin{alignedat}{2}
\mbox{either } \ & u\in H^1(\Omega)\cap H^s (\Omega)  \quad&&\mbox{ if } \lambda<\lambda_1,\\
\mbox{or } \ &u\in H^1(\Omega)\cap H^s (\Omega)\cap L^2(\R^N) \quad&&\mbox{ if } \lambda\ge\lambda_1.
\end{alignedat}
\end{equation}

If instead~$\alpha=0$, for any $\lambda\in\R$, the problem
\begin{equation}\label{P12}
\begin{cases}
\beta(-\Delta)^s u +u=\lambda u +f(x,u)  &\mbox{in }\Omega,\\
 \Ns u(x)=0 &\mbox{for all }x\in \R^N \setminus \overline{\Omega}, 
\end{cases}
\end{equation}
admits a nontrivial solution $u$ such that
\begin{equation}\label{spazicor11}
\begin{alignedat}{2}
\mbox{either } \ & u\in H^s (\Omega)  \quad&&\mbox{ if } \lambda<\lambda_1,\\
\mbox{or } \ &u\in \widetilde{\mathcal H}_{0, \mu} (\Omega) \quad&&\mbox{ if } \lambda\ge\lambda_1,
\end{alignedat}
\end{equation}
being 
\[\
\widetilde{\mathcal H}_{0, \mu}(\Omega):=\left\{u\in H^s(\Omega) : \Ns (u)=0 \mbox{ in }\R^N\setminus \overline{\Omega}\right\}.
\]
\end{corollary}
\begin{proof}
Let $\mu:=\delta_s$, where~$\delta_s$ is the Dirac's delta at~$s$.
Recalling~\eqref{ssharp}, we have that~$s_\sharp=1$ if~$\alpha\neq0$ and~$s_\sharp=s$ if~$\alpha=0$.
Hence, assumption~\eqref{AR1} is verified by taking either~$p\in (2, 2^*)$ if~$\alpha\neq 0$ or~$p\in (2, 2^*_s)$ if~$\alpha =0$.

The existence of a nontrivial solution for problems~\eqref{P11} and~\eqref{P12} is guaranteed by the Theorems~\ref{MPT} and~\ref{LINK}.
Moreover, we have that~\eqref{spazicor1} and~\eqref{spazicor11} plainly follow by the definition of the spaces in~\eqref{HALPHAMU} and~\eqref{tildeacca}, respectively.
\end{proof}

One more interesting application of our results arises when we choose~$\mu$ to be a finite (or even infinite) sum of Dirac's measures.  On this topic, we provide the following two results.
To our best knowledge,  these results are entirely new in the literature.

\begin{corollary}\label{cor2}
Let $n\in\N$ with $n\ge 2$ and $s_k\in (0, 1)$ for any $k\in\{1, \dots, n\}$.  Denote by $\widetilde\lambda_k$ the sequence of the eigenvalues of the operator
\[
\-\alpha\Delta u  + \displaystyle\sum_{k=1}^n(-\Delta)^{s_k} u
\]
under homogeneous Neumann conditions.

Then, if $\alpha\neq 0$, for any $\lambda\in\R$, the problem
\begin{equation}\label{P21}
\begin{cases}
-\alpha\Delta u  + \displaystyle\sum_{k=1}^n(-\Delta)^{s_k} u +u=\lambda u +f(x,u)  &\mbox{in }\Omega,\\
\partial_\nu u(x)=0  &\mbox{for all } x\in \partial\Omega,\\
\displaystyle\sum_{k=1}^n\mathscr N_{s_k} u(x)=0 &\mbox{for all }x\in \R^N \setminus \overline{\Omega}, 
\end{cases}
\end{equation}
admits a nontrivial solution $u$ such that
\begin{equation}\label{spazicor2}
\begin{alignedat}{2}
\mbox{either } \ & u\in H^1(\Omega)\cap H^{\overline s} (\Omega)  \quad&&\mbox{ if } \lambda<\lambda_1,\\
\mbox{or } \ &u\in H^1(\Omega)\cap H^{\overline s} (\Omega)\cap L^2(\R^N) \quad&&\mbox{ if } \lambda\ge\lambda_1,
\end{alignedat}
\end{equation}
where $\overline s:=\max\{s_k : k\in\{1,\dots, n\}\}$.

If instead~$\alpha=0$, for any $\lambda\in\R$, the problem
\begin{equation}\label{P22}
\begin{cases}
\displaystyle\sum_{k=1}^n (-\Delta)^{s_k} u +u=\lambda u +f(x,u)  &\mbox{in }\Omega,\\
\displaystyle\sum_{k=1}^n\mathscr N_{s_k} u(x)=0 &\mbox{for all }x\in \R^N \setminus \overline{\Omega}, 
\end{cases}
\end{equation}
admits a nontrivial solution $u$ such that
\begin{equation}\label{spazicor22}
\begin{alignedat}{2}
\mbox{either } \ & u\in H^{\overline s} (\Omega)  \quad&&\mbox{ if } \lambda<\lambda_1,\\
\mbox{or } \ &u\in \widetilde{\mathcal H}_{0, \mu} (\Omega) \quad&&\mbox{ if } \lambda\ge\lambda_1,
\end{alignedat}
\end{equation}
being
\[\
\widetilde{\mathcal H}_{0, \mu}(\Omega):=\left\{u\in H^{\overline s}(\Omega) : \sum_{k=1}^n\mathscr N_{s_k} (u)=0 \mbox{ in }\R^N\setminus \overline{\Omega}\right\}.
\]
\end{corollary}

\begin{proof}
Let $$\mu:=\displaystyle\sum_{k=1}^n \delta_{s_k}.$$ By~\eqref{ssharp}, we have that~$s_\sharp=1$ if~$\alpha\neq0$ and~$s_\sharp = \overline s$ if~$\alpha=0$. In particular, we have that assumption~\eqref{AR1} is verified by taking either~$p\in (2, 2^*)$ if~$\alpha\neq 0$ or~$p\in (2, 2^*_{\overline s})$ if~$\alpha =0$.

Then, the desired nontrivial solutions for problems~\eqref{P21} and~\eqref{P22} are deduced by Theorems~\ref{MPT} and~\ref{LINK}.
In addition, \eqref{spazicor2} and~\eqref{spazicor22} are a consequence of definitions~\eqref{HALPHAMU} and~\eqref{tildeacca}, respectively. 
\end{proof}

\begin{corollary}\label{cor3}
Let $c_k\ge 0$ and $s_k\in (0, 1)$ for any $k\in\N\setminus\{0\}$. 
Suppose that
$$ \sum_{k=0}^{+\infty} c_k\in(0,+\infty).$$
Denote by $\widetilde\lambda_k$ the sequence of the eigenvalues of the operator
\[
-\alpha\Delta  u + \displaystyle\sum_{k=1}^{+\infty} c_k (-\Delta)^{s_k} u
\]
under homogeneous Neumann conditions.

Then, if $\alpha\neq 0$, for any $\lambda\in\R$, the problem
\begin{equation}\label{P31}
\begin{cases}
-\alpha\Delta  u + \displaystyle\sum_{k=1}^{+\infty} c_k (-\Delta)^{s_k} u +u=\lambda u +f(x,u)  &\mbox{in }\Omega,\\
\partial_\nu u(x)=0  &\mbox{for all } x\in \partial\Omega,\\
\displaystyle\sum_{k=1}^{+\infty} c_k \mathscr N_{s_k} u(x)=0 &\mbox{for all }x\in \R^N \setminus \overline{\Omega}, 
\end{cases}
\end{equation}
admits a nontrivial solution $u$ such that
\begin{equation}\label{spazicor3}
\begin{alignedat}{2}
\mbox{either } \ & u\in H^1(\Omega)\cap H_\mu(\Omega)  \quad&&\mbox{ if } \lambda<\lambda_1,\\
\mbox{or } \ &u\in H^1(\Omega)\cap H_{\mu}(\Omega)\cap L^2(\R^N) \quad&&\mbox{ if } \lambda\ge\lambda_1,
\end{alignedat}
\end{equation}
being\footnote{We notice that if there exists $\overline k\in\N\setminus\{0\}$ such that $s_{\overline k}\ge s_k$ for any $k\in\N\setminus\{0\}$, then $H_\mu(\Omega)$ turns out to be the space $H^{s_{\overline k}}(\Omega)$.}
\[
H_\mu (\Omega):=\bigcap_{k=1}^{+\infty} H^{s_k}(\Omega).
\]

If instead~$\alpha=0$, for any $\lambda\in\R$, the problem
\begin{equation}\label{P32}
\begin{cases}
\displaystyle\sum_{k=1}^{+\infty} c_k(-\Delta)^{s_k} u +u=\lambda u +f(x,u)  &\mbox{in }\Omega,\\
\displaystyle\sum_{k=1}^{+\infty} c_k \mathscr N_{s_k} u(x)=0 &\mbox{for all }x\in \R^N \setminus \overline{\Omega}, 
\end{cases}
\end{equation}
admits a nontrivial solution $u$ such that
\begin{equation}\label{spazicor33}
\begin{alignedat}{2}
\mbox{either } \ & u\in H_\mu (\Omega)  \quad&&\mbox{ if } \lambda<\lambda_1,\\
\mbox{or } \ &u\in \widetilde{\mathcal H}_{0, \mu} (\Omega) \quad&&\mbox{ if } \lambda\ge\lambda_1,
\end{alignedat}
\end{equation}
being
\[\
\widetilde{\mathcal H}_{0, \mu}(\Omega):=\left\{u\in H_{\mu}(\Omega) : \sum_{k=1}^{+\infty}\mathscr N_{s_k} (u)=0 \mbox{ in }\R^N\setminus \overline{\Omega}\right\}.
\]
\end{corollary}
\begin{proof}
Let 
$$ \mu:= \sum_{k=0}^{+\infty} c_k \,\delta_{s_k}.$$
We notice that two cases can occur. Indeed, recalling~\eqref{ssharp}, if
$$\displaystyle\overline s:=\sup_{k\in\N\setminus\{0\}}\{s_k\} \neq s_k$$ for every~$k$, then we can take~$s_\sharp:= \overline s-\delta$, for some~$\delta>0$ sufficiently small.
If instead~$\overline s = s_{\overline k }$ for some~$\overline k$, then we can take~$s_\sharp:= s_{\overline k}$.

The desired nontrivial solutions for problems~\eqref{P31} and~\eqref{P32} are inferred by Theorems~\ref{MPT} and~\ref{LINK}.
Moreover, \eqref{spazicor3} and~\eqref{spazicor33} come from definitions~\eqref{HALPHAMU} and~\eqref{tildeacca}, respectively. 
\end{proof}

Remarkably,  another interesting result stems from the continuous superposition of fractional operators. To the best of our knowledge,  this result turns out to be new as well.

\begin{corollary}\label{cor4}
Let $\omega:(0, 1)\to\R^+$ be a measurable, nonnegative and not identically zero function.  Denote by $\widetilde\lambda_k$ the sequence of the eigenvalues of the operator
\begin{equation}\label{intoperator}
-\alpha\Delta  u + \displaystyle\int_0^1 \omega(s) (-\Delta)^s u\, ds
\end{equation}
under homogeneous Neumann conditions.

Then, if~$\alpha\neq 0$, for any $\lambda\in\R$, the problem
\begin{equation}\label{P41}
\begin{cases}
-\alpha\Delta  u + \displaystyle\int_0^1 \omega(s) (-\Delta)^s u\, ds +u=\lambda u +f(x,u)  &\mbox{in }\Omega,\\
\partial_\nu u(x)=0  &\mbox{for all } x\in \partial\Omega,\\
\displaystyle\int_0^1 \omega(s) \Ns u\, ds=0 &\mbox{for all }x\in \R^N \setminus \overline{\Omega}, 
\end{cases}
\end{equation}
admits a nontrivial solution $u$ such that
\begin{equation}\label{spazicor4}
\begin{alignedat}{2}
\mbox{either } \ & u\in H^1(\Omega)\cap \mathcal H_\mu (\Omega) \quad&&\mbox{ if } \lambda<\lambda_1,\\
\mbox{or } \ &u\in H^1(\Omega)\cap \mathcal H_\mu (\Omega)\cap L^2(\R^N) \quad&&\mbox{ if } \lambda\ge\lambda_1,
\end{alignedat}
\end{equation}
being 
\[
\mathcal H_\mu (\Omega):= \left\lbrace u:\R^N\to\R \text{ measurable } : u\in L^2(\Omega) \ \mbox{ and } \ \int_0^1 \omega(s) [u]_s^2 \, ds <+\infty \right\rbrace.
\]

If instead~$\alpha=0$, for any $\lambda\in\R$, the problem
\begin{equation}\label{P42}
\begin{cases}
\displaystyle\int_0^1 \omega(s) (-\Delta)^s u\, ds +u=\lambda u +f(x,u)  &\mbox{in }\Omega,\\
\displaystyle\int_0^1 \omega(s) \Ns u\, ds=0  &\mbox{for all }x\in \R^N \setminus \overline{\Omega}, 
\end{cases}
\end{equation}
admits a nontrivial solution $u$ such that
\begin{equation}\label{spazicor44}
\begin{alignedat}{2}
\mbox{either } \ & u\in \mathcal H_\mu (\Omega)  \quad&&\mbox{ if } \lambda<\lambda_1,\\
\mbox{or } \ &u\in \widetilde{\mathcal H}_{0, \mu} (\Omega) \quad&&\mbox{ if } \lambda\ge\lambda_1,
\end{alignedat}
\end{equation}
being
\[\
\widetilde{\mathcal H}_{0, \mu}(\Omega):=\left\{u\in \mathcal H_{\mu}(\Omega) :\displaystyle\int_0^1 \omega(s) \Ns u\, ds=0 \mbox{ in }\R^N\setminus \overline{\Omega}\right\}.
\]
\end{corollary}

\begin{proof}
We observe that the operator in~\eqref{intoperator} is a particular case of $L_{\alpha, \mu}$ defined in~\eqref{ELLE}, where~$d\mu(s)$ reduces to~$\omega(s)ds$.

In this setting, recalling~\eqref{ssharp}, if~$\alpha\neq0$ we take~$s_\sharp=1$.
If instead~$\alpha=0$, we can take~$s_\sharp$ to be any value in~$(0, 1)$ satisfying
\[
\int_{s_\sharp}^1 \omega(s) \, ds >0.
\]
This choice makes $s_\sharp$ to play the role of critical exponent in~\eqref{AR1}.

The desired result for both problems~\eqref{P41} and~\eqref{P42} follows by Theorems~\ref{MPT} and~\ref{LINK}.
Moreover, \eqref{spazicor4} and~\eqref{spazicor44} are a consequence of definitions~\eqref{HALPHAMU} and~\eqref{tildeacca}, respectively. 
\end{proof}

\begin{appendix}
\section{Some justification for the choice of $\widetilde{\mathcal H}_{\alpha, \mu}(\Omega)$}
\label{sec-app}

Here we motivate our choice of defining the space~$\widetilde{\mathcal H}_{\alpha, \mu}(\Omega)$ as in~\eqref{tildeacca}. 

The reason is that, learning from the case $\alpha=0$, one may be tempted to consider instead the space
\[
\overline{\mathcal H}_{\alpha, \mu}(\Omega)
:=\left\{u\in\mathcal H_{\alpha, \mu}(\Omega) : \;\alpha\dfrac{\partial u}{\partial\nu}=0
\mbox{ on } \partial\Omega
 \quad\mbox{ and }\quad \int_{(0, 1)}\Ns u\, d\mu(s)=0
\mbox{ in }\R^N\setminus \overline{\Omega} \right\}
\]
even when $\alpha\neq 0$.
Following this definition, we can find a  complete orthogonal system 
for~$\overline{\mathcal H}_{\alpha, \mu}(\Omega)$ as we did for~$\widetilde{\mathcal H}_{0, \mu}(\Omega)$ in
Proposition~\ref{sistemacompleto}, that is,
from the sequence of eigenfunctions provided in~\cite[Theorem~1.5]{TUTTI}.  Indeed, the following result holds:

\begin{proposition}
The sequence $e_k$ is a complete orthogonal system for~$\overline{\mathcal H}_{\alpha, \mu}(\Omega)$.
\end{proposition}

\begin{proof}
The orthogonality in $\overline{\mathcal H}_{\alpha, \mu}(\Omega)$
comes from the orthogonality in $L^2(\Omega)$. Indeed,
for any~$k$, $j\in \N$, by construction we have that
\[
\alpha\int_\Omega \nabla e_k\cdot \nabla e_j\,dx
+\int_{(0,1)}\frac{c_{N, s}}{2}\iint_\Q\frac{(e_k(x)-e_k(y))(e_j(x)-e_j(y))}{|x-y|^{N+2s}}\,dx\,dy\,d\mu(s)=\lambda_k \int_\Omega e_k e_j\,dx.
\]
To prove the completeness, let $u\in \overline{\mathcal H}_{\alpha, \mu}(\Omega)$. {F}rom the completeness in $L^2(\Omega)$, if we 
consider the restriction of $u$ in $\Omega$, we can find coefficients~$a_k$ such that
\[
u=\sum_{k=1}^{+\infty}a_ke_k \quad \mbox{a.e. in } \Omega
.\]
Moreover, since $u\in \overline{\mathcal H}_{\alpha, \mu}(\Omega)$,
when $x\in \R^N\setminus \overline{\Omega}$ we have
\[
\begin{aligned}
u(x)&=
\dfrac{\displaystyle \int_{(0,1)}\int_\Omega \frac{u(y)}{|x-y|^{1+2s}}\,dy\,d\mu(s)}{\displaystyle \int_{(0,1)}\int_\Omega \frac{dy}{|x-y|^{1+2s}}\,d\mu(s)} \\
&=\dfrac{\displaystyle \int_{(0,1)}\int_\Omega \frac{\sum_{k=1}^{+\infty}a_ke_k(y)}{|x-y|^{1+2s}}\,dy\,d\mu(s)}{\displaystyle \int_{(0,1)}\int_\Omega \frac{dy}{|x-y|^{1+2s}}\,d\mu(s)} \\
&=\sum_{k=1}^{+\infty}a_k\dfrac{\displaystyle \int_{(0,1)}\int_\Omega \frac{e_k(y)}{|x-y|^{1+2s}}\,dy\,d\mu(s)}{\displaystyle \int_{(0,1)}\int_\Omega \frac{dy}{|x-y|^{1+2s}}\,d\mu(s)} \\
&=\sum_{k=1}^{+\infty}a_k e_k(x),
\end{aligned}          
\]
which completes the proof.
\end{proof}

However, the space $\overline{\mathcal H}_{\alpha, \mu}(\Omega)$ is not closed with respect to the weak convergence in~$\overline{\mathcal H}_{\alpha, \mu}(\Omega)$, not even in the local case when~$\mu\equiv 0$, as illustrated in the example below.

\begin{example}
Let $N=1$, $\Omega=(0,1)$ and $\mu\equiv 0$. In this case, the space
$\overline{\mathcal H}_{\alpha, \mu}(\Omega)$ boils down to
\[
\overline{\mathcal H}_{\alpha, 0}(\Omega)
=\left\{u\in H^1((0,1)):u'(0)=u'(1)=0 \right\}.
\]
Let $u_n$ be a sequence in $H^1((0,1))$ defined as
\[
u_n(x):=x^{1+\frac{1}{n}}(x-1)^2.
\]
Clearly, for any $n\in \N$,
\[
u'_n(x)=x^\frac{1}{n}(x-1)\left[\left(3+\frac{1}{n}\right)x -1-	\frac{1}{n}\right],
\]
and thus the sequence $u_n$ is in $\overline{\mathcal H}_{\alpha, 0}(\Omega)$. Moreover, we have that
\begin{eqnarray*}
&&\lim_{n\to+\infty}u_n(x) = x(x-1)^2=:u(x) \quad \mbox{a.e. in }(0,1) \\
{\mbox{and }}\quad && \lim_{n\to+\infty}
u'_n(x) = (x-1)(3x-1)=u'(x) \quad \mbox{a.e. in }(0,1).
\end{eqnarray*}

Now, we show that
\begin{equation}\label{843753dfsvfywqugewquyt}
{\mbox{$u_n$ converges strongly to $u$ in $H^1((0,1))$ as~$n\to+\infty$.}}\end{equation}
Indeed, 
\begin{eqnarray*}&&
\|u_n-u\|_{L^2((0,1))}^2=
\int_0^1\big|x^{1+\frac{1}{n}}(x-1)^2-x(x-1)^2\big|^2\,dx
\leq \int_0^1|x^{\frac{1}{n}}-1|^2\,dx \\
&&\qquad\qquad =\left(1+\frac{2}{n}\right)^{-1}-2\left(1+\frac{1}{n}\right)^{-1}
+1,
\end{eqnarray*}
which entails that
\begin{equation}\label{s3264753cascvbtry}
\lim_{n\to +\infty}\|u_n-u\|_{L^2((0,1))}=0.
\end{equation}

In addition,
\[
\begin{aligned}
\|u'_n-u'\|_{L^2((0,1))}&=
\int_0^1\left|x^\frac{1}{n}(x-1)\left[\left(3+\frac{1}{n}\right)x -1-	\frac{1}{n}\right]-(x-1)(3x-1)\right|^2\,dx \\
&\leq \int_0^1\left|\left[\left(3+\frac{1}{n}\right)x^{1+\frac{1}{n}} -3x\right]-\left[\left(1+\frac{1}{n}\right)x^{\frac{1}{n}}-1 \right]\right|^2\,dx \\
&=\left(3+\frac{1}{n}\right)^2\left(3+\frac{2}{n}\right)^{-1}
+\left(1+\frac{1}{n}\right)^2\left(1+\frac{2}{n}\right)^{-1}
+2\left(3+\frac{1}{n}\right)\left(2+\frac{1}{n}\right)^{-1} \\
&\qquad 
+6\left(1+\frac{1}{n}\right)\left(2+\frac{1}{n}\right)^{-1}
-2\left(3+\frac{1}{n}\right)\left(1+\frac{1}{n}\right)
\left(2+\frac{1}{n}\right)^{-1}-7.
\end{aligned}
\]
Therefore,
\[
\lim_{n\to +\infty}\|u'_n-u'\|_{L^2((0,1))}=0.
\]
This and~\eqref{s3264753cascvbtry} give~\eqref{843753dfsvfywqugewquyt}.

In particular, we have that~$u_n$ converges weakly to~$u$ in~$H^1((0,1))$ as~$n\to+\infty$,
but $u\not \in \overline{\mathcal H}_{\alpha, 0}(\Omega)$ as~$u'(0)=1\neq 0$.

This means that~$\overline{\mathcal H}_{\alpha, 0}(\Omega)$ is not closed with respect to the weak convergence in~$\overline{\mathcal H}_{\alpha, 0}(\Omega)$.
\end{example}

\end{appendix}

\section*{Acknowledgements} 
All the authors are members of the Australian Mathematical Society (AustMS). CS and EPL are members of the INdAM--GNAMPA.

This work has been supported by the Australian Laureate Fellowship FL190100081.

\vfill

\end{document}